\documentclass{amsart}

\usepackage{amssymb, amsmath,  amsfonts, graphicx} %, mathabx}
\usepackage{color}
 \usepackage{amsrefs}

\newtheorem{thm}{Theorem}[section]
\newtheorem{prop}[thm]{Proposition}
\newtheorem{lemma}[thm]{Lemma}
\newtheorem{cor}[thm]{Corollary}
\newtheorem{defn}[thm]{Definition}%[section]
\newtheorem{lem-defn}[thm]{Lemma/Definition}%[section]
\newtheorem{defn-eg}[thm]{Definition/Example}%[section]

\newtheorem{example}[thm]{Example}
\newtheorem{remark}[thm]{Remark}

\newtheorem{prop-defn}[thm]{Proposition/Definition}%[section]
\newcommand{\circline}{\!\;\!\!\circ\!\!\;\!\!-\!\!\!-\!}

  \numberwithin{equation}{section}

\usepackage{graphicx}
\makeatletter
\DeclareRobustCommand\widecheck[1]{{\mathpalette\@widecheck{#1}}}
\def\@widecheck#1#2{%
   \setbox\z@\hbox{\m@th$#1#2$}%
   \setbox\tw@\hbox{\m@th$#1%
      \widehat{%
         \vrule\@width\z@\@height\ht\z@
         \vrule\@height\z@\@width\wd\z@}$}%
   \dp\tw@-\ht\z@
   \@tempdima\ht\z@ \advance\@tempdima2\ht\tw@ \divide\@tempdima\thr@@
   \setbox\tw@\hbox{%
      \raise\@tempdima\hbox{\scalebox{1}[-1]{\lower\@tempdima\box\tw@}}}%
   {\ooalign{\box\tw@ \cr \box\z@}}}
\makeatother

\begin{document}{\allowdisplaybreaks[4]

\title{On equivariant quantum Schubert calculus for $G/P$}

\author{Yongdong Huang}
\address{Department of Mathematics,
           Jinan University, Guangzhou, Guangdong, China}
%\curraddr{Department of Mathematics }
\email{tyongdonghuang@jnu.edu.cn}
\thanks{ %The first author  is supported in part by %a RGC research grant from the Hong Kong Government.
%The second author  is  supported in part %by KRF-2007-341-C00006.
 }


\author{Changzheng Li}
\address{Center for Geometry and Physics, Institute for Basic Science (IBS), Pohang 790-784, Republic of   Korea}
\email{czli@ibs.re.kr}

\date{%January 1, 2001 and, in revised form, June 22, 2001.
      }


%\keywords{Lie superalgebra, spin  representations}

%\subtitle{}

\begin{abstract}
 We show a $\mathbb{Z}^2$-filtered algebraic structure and  a ``quantum to classical" principle  on  the torus-equivariant quantum cohomology of a complete flag variety of general Lie type, generalizing earlier works of Leung and the second author. We also provide various applications on equivariant quantum Schubert calculus, including
  an equivariant quantum Pieri rule for  partial flag variety  $F\ell_{n_1, \cdots, n_k; n+1}$ of Lie type $A$.
 \end{abstract}

\maketitle
%%\tableofcontents

\section{Introduction}
The complex Grassmannian $Gr(m, n+1)$ parameterizes    $m$-dimensional complex
vector subspaces of $\mathbb{C}^{n+1}$.
The integral  cohomology ring $H^*(Gr(m, n+1), \mathbb{Z})$   has  an additive   basis of Schubert classes $\sigma^\nu$, indexed by partitions $\nu=(\nu_1, \cdots, \nu_m)$ inside an $m\times (n+1-m)$ rectangle:   $n+1-m\geq \nu_1\geq  \cdots\geq \nu_m\geq 0$.
The initial classical Schubert calculus, in modern language,   refers to the study of   the ring structure of $H^*(Gr(m, n+1), \mathbb{Z})$. The content includes
\begin{enumerate}
  \item a Pieri rule,
   giving a %(visibly positive)
   combinatorial formula for the cup product by a set of generators of the cohomology ring, for instance by the special Schubert classes $\sigma^{1^p}$, $1\leq p\leq m$, where  $1^p=(1,\cdots, 1, 0, \cdots, 0)$ has precisely $p$ copies of 1;
  \item more generally, a Littlewood-Richardson rule,  giving a (manifestly positive) combinatorial formula of the structure constants $N_{\mu, \nu}^\eta$ in the cup product
          $\sigma^\mu\cup \sigma^\nu=\sum_{\eta}N_{\mu, \nu}^\eta \sigma^\eta $;
  \item a ring presentation of $H^*(Gr(m, n+1), \mathbb{Z})$;
  \item  a Giambelli formula, expressing every  $\sigma^\nu$ as a polynomial in special Schubert classes.
\end{enumerate}
The complex Grassmannian $Gr(m, n+1)$ %%can be written as a quotient of $SL(n+1, \mathbb{C})$ by a subgroup of block-upper triangular matrices with precisely two diagonal blocks  of size $m\times m$ and $(n+1-m)\times (n+1-m)$.
 is a special case of homogeneous varieties $G/P$, where  $G$ denotes the adjoint group of a  complex simple Lie algebra of rank $n$, and $P$ denotes  a parabolic subgroup of $G$.
The classical Schubert calculus, in  general,
    refers to the study of the classical cohomology ring $H^*(G/P)=H^*(G/P, \mathbb{Z})$.

There are various   extensions of the classical Schubert calculus by replacing ``classical" with ``equivariant", ``quantum", or ``equivariant quantum". The equivariant quantum Schubert calculus for $G/P$  refers to the study  of the (integral) torus-equivariant quantum cohomology ring $QH^*_T(G/P)$,  % on the equivariant quantum extension of (1), (2), (3) and (4).
which %$QH^*_T(G/P)$
  is a deformation of the ring structure of the torus-equivariant cohomology $H^*_T(G/P)$ by incorporating genus zero, three-point equivariant Gromov-Witten invariants.
 In analogy with $H^*(Gr(m, n+1))$, the ring $QH^*_T(G/P)$ has  a basis of Schubert classes $\sigma^u$ over $H^*_T(\mbox{pt})[q_1, \cdots, q_k]$  where $k:=\dim H_2(G/P)$, indexed by elements in a subset $W^P$ of the   Weyl group $W$ of $G$.
The structure coefficients $N_{u, v}^{w, \mathbf{d}}$ of the equivariant quantum product
 $$\sigma^u\star \sigma^v=\sum_{w\in W^P, \mathbf{d}\in H_2(G/P, \mathbb{Z})}N_{u, v}^{w, \mathbf{d}}  \sigma^w q^{\mathbf{d}}  $$
 are homogeneous polynomials in   $H_{T}^*(\mbox{pt})=\mathbb{Z}[\alpha_1, \cdots, \alpha_n]$ with variables $\alpha_i$ being   simple roots of $G$.
They contain all the  information in      the former kinds of Schubert calculus. For instance, the non-equivariant limit of $N_{u, v}^{w, \mathbf{d}}$, given by evaluating $(\alpha_1,\cdots,  \alpha_n)=\mathbf{0}$,  recovers an ordinary  Gromov-Witten invariant, which counts the number of  degree $\mathbf{d}$ rational curves in $G/P$ passing through three Schubert subvarieties associated to $u, v$, $w$.

When $G=PSL(n+1, \mathbb{C})$, the Weyl group $W$ is a permutation group $S_{n+1}$ generated by transpositions $s_i=(i, i+1)$. Every   homogenous variety   $PSL(n+1, \mathbb{C})/P$ is of the form  $F\ell_{n_1, \cdots, n_k; n+1}:=\{V_{n_1}\leqslant\cdots \leqslant V_{n_k}\leqslant \mathbb{C}^{n+1}~|~ \dim_{\mathbb{C}}V_{n_j}=n_j, \,\,\forall 1\leq j\leq k\}$, parameterizing  partial flags in $\mathbb{C}^{n+1}$.
As an algebra over $H_T^*(\mbox{pt})[q_1, \cdots, q_k]$, the equivariant quantum cohomology ring   $QH^*_T(F\ell_{n_1, \cdots, n_k; n+1})$  is generated (see e.g. \cite{AnCh,LaSh-QDSP}) by  special Schubert classes
$$  \sigma^{c[n_i, p]},\quad\mbox{where}\quad c[n_i, p]:=s_{n_i-p+1}\cdots s_{n_i-1}s_{n_i}.$$
One of the main results of our present paper is the following equivariant quantum Pieri rule. The non-equivariant limit of it recovers the  quantum Pieri rule, which was first given by   Ciocan-Fontanine \cite{CFon}, and was reproved by Buch \cite{Buch-PartialFlag}. The classical limit (by evaluating $\mathbf{q}=\mathbf{0}$) is a slight improvement of Robinson's equivariant Pieri rule \cite{Robinson}. By evaluating all equivariant parameters $\alpha_i$ and all quantum parameters $q_j$ at 0, we have  the classical Pieri rule due to Lascoux and Sch\"utzenberger \cite{LasSch} and  Sottile \cite{Sottile}.
\vspace{0.2cm}

\noindent\textbf{Theorem \ref{thmequiQPR}.}  {\itshape For any $1\leq i\leq k, 1\leq p\leq n_i$ and any $u\in W^P$, we have
{\upshape $$ \sigma^{c[n_i, p]}\star \sigma^u=\sum_{\mathbf{d}\in \scriptsize\mbox{Pie}_{i, p}(u)}  \sum_{j=0}^{p-d_i}  \sum_{w} \xi^{n_i-d_i-j,  p-d_i-j}(\mu_{w\cdot \phi_\mathbf{d}, u\cdot \tau_\mathbf{d}, n_i-d_i})\sigma^w q^{\mathbf{d}}$$
} with the last summation over those {\upshape $w\in \mbox{Per}(\mathbf{d})$} satisfying $w\cdot \phi_{\mathbf{d}}\in S_{n_i-d_i, j}(u\cdot \tau_\mathbf{d})$.

}
\vspace{0.2cm}

\noindent Here   $\mbox{Pie}_{i, p}(u)$, etc., are combinatorial sets to be described    in section \ref{subsecequivFlagAtype};  the element
  $\mu=\mu_{w\cdot \phi_\mathbf{d}, u\cdot \tau_\mathbf{d}, n_i-d_i}$  is a partition inside an $(n_i-d_i-j)\times (n+1-n_i+d_i+j)$ rectangle. Each structure coefficient
 $\xi^{n_i-d_i-j,  p-d_i-j}(\mu)$ coincides with
  the coefficient of $\sigma^\mu$ in the equivariant product $ \sigma^{c[n_i, p]}\circ \sigma^\mu$ in
  $H^*_T(Gr(n_i-d_i-j, n+1))$.
  Geometrically, it is the restriction of the Grassmannian Schubert class $\sigma^{(1^{p-d_i-j})}$ of $H^*_T(Gr(n_i-d_i-j, n+1))$ (labeled by the special partition of $(p-d_i-j)$ copies of 1) to a $T$-fixed point labeled by the partition $\mu$.
 These restriction type structure coefficients  can be easily computed in many ways \cite{KoKu,Billey,BuRi,KnutTao,Laksov,GaSa,Kreiman,ThYo}.
  For completeness, we will include a known formula in Definition \ref{defnKosPolyforCoeff}.
 We remark that our  formula above is different from the one in   \cite{LaSh-affinePieri} by Lam and Shimozono, which concerns  the multiplications by  $\sigma^{s_ps_{p-1}\cdots s_2s_1s_\theta}$ in the ring
  $QH^*_T(F\ell_{1, 2, \cdots, n; n+1})$. Here  $\theta$ denotes the highest root, and a ring isomorphism  \cite{Peterson,LaSh-GoverPaffineGr,LeungLi-GWinv} between  $QH^*_T(F\ell_{1, 2, \cdots, n; n+1})$ and
     the equivariant homology $H_*^T(\Omega SU(n+1))$ of based loop groups after localization  is involved.

In the special case of complex Grassmannians, $H_2(Gr(m, n+1), \mathbb{Z})=\mathbb{Z}$, so that there is only one quantum variable $q$.
 Let $\mathcal{P}_{m, n+1}$ denote the set of partitions inside an $m\times (n+1-m)$ rectangle.
  For   $\nu=(\nu_1, \cdots, \nu_m)$ and $\eta=(\eta_1, \cdots, \eta_m)$ in $\mathcal{P}_{m, n+1}$   satisfying $\eta_i-\nu_i\in\{0, 1\}$ for all $i$, we introduce an associated  partition $\eta_\nu\in\mathcal{P}_{m-r, n+1}$,
    $$\eta_\nu:=(\nu_{j_1}-j_1+r+1,  \nu_{j_2}-j_2+r+2,\cdots, \nu_{j_{m-r}}-j_{m-r}+m),$$
where   $j_1<j_2<\cdots <j_{m-r}$ denote  all  entries with $\eta_{j_i}=\nu_{j_i}$. In other words, the Young diagram of $\eta$ is obtained by adding a vertical strip to the Young diagram of $\nu$; the associated partition $\eta_\nu$ can also described with this flavor by using a simple join-and-cut operation (see Definition \ref{defeg} and the figures therein for more details).  A further simplification of the above theorem leads to  the following  equivariant quantum Pieri rule for complex Grassmannians.

\vspace{0.2cm}

 \noindent\textbf{Theorem \ref{thmQPRforComplexGrassmannian}.}  {\itshape
   Let $1\leq p\leq m$ and  $\nu=(\nu_1, \cdots, \nu_m)\in \mathcal{P}_{m, n+1}$. In $QH^*_T(Gr(m, n+1))$,
     $$\sigma^{1^p}\star \sigma^\nu=  \sum_{r=0}^{p} \sum_{\eta} \xi^{m-r,  p-r}(\eta_\nu)\sigma^\eta+
      \sum_{r=0}^{p-1}  \sum_{\kappa} \xi^{m-1-r,  p-1-r}(\kappa'_{\nu'})\sigma^\kappa q,$$
 where the second sum is over partitions $\eta=(\eta_1, \cdots, \eta_m)\in\mathcal{P}_{m, n+1}$   satisfying $|\eta|=|\nu|+r$ and $\eta_i-\nu_i\in\{0, 1\}$ for all $i$;  the $q$-terms
 occur  only if $\nu_1=n+1-m$, and when this holds,   the last sum is over
    partitions $\kappa=(\kappa_1, \cdots, \kappa_{m-1},0)$ such that
      $\kappa':=(\kappa_1+1, \cdots, \kappa_{m-1}+1)$ and $\nu':=(\nu_2, \cdots, \nu_m)$
        satisfy $|\kappa'|=|\nu'|+r$ and $\kappa_i+1-\nu_{i+1}\in\{0, 1\}$ for all $1\leq i\leq m-1$.
}

\vspace{0.2cm}

\noindent For instance  in $QH^*_T(Gr(3, 7))$, we have
    $$\sigma^{(1,1,1)}\star \sigma^{(4,0,0)} =(\alpha_1+\cdots+\alpha_6)\sigma^{(4, 1,  1)}+q.$$
The special case   when $p=1$ has been given by Mihalcea  \cite{Miha-EQSC}.   The full  Pieri-type formula   could have been obtained by combining the   study of the equivariant quantum K-theory \cite{BuMi-quantumKtheory}  with an equivariant Pieri rule  \cite{Laksov,GaSa} using Grassmannian algebras. It might also  be deduced  by the   results in   \cite{LaSh-GoverPaffineGr,LaSh-affinePieri}. However,   an explicit statement, which is a very important component of the equivariant quantum Schubert calculus, has not been given anywhere else yet. Actually, the equivariant Pieri rule, which is read off from the classical part of Theorem \ref{thmQPRforComplexGrassmannian}, is different from the aforementioned known formulations. It has even inspired the second author and Ravikumar to find an equivariant Pieri rule for Grassmannians of all classical Lie types with a geometric approach in their recent work \cite{LiRa}.   On the other hand, Buch has recently shown  a   (manifestly positive) equivariant puzzle rule for two-step flag varieties \cite{Buch-equivTwostep}, generalizing the rules in \cite{KnutTao, BKPT}. Therefore he obtains a Littlewood-Richardson rule for $QH^*_T(Gr(m, n+1))$, due to the equivariant ``quantum to classical" principle \cite{BuMi-quantumKtheory}. It will be interesting to study how we simplify  Buch's general rule in the special case of Pieri rule to obtain a more compact form as above.

In analogy with the contents (1),(2),(3),(4) of the classical Schubert calculus, there are the corresponding  equivariant quantum extensions, say $(1)', (2)', (3)'$ and $(4)'$,  in the   equivariant quantum Schubert calculus. The problem $(2)'$ of finding a manifestly  positive formula of the structure constants for the equiavariant quantum cohomology remains open except for very few cases including complex Grassmannians  \cite{Buch-equivTwostep} (which is widely open even for the classical Schubert calculus). For $(3)'$ and $(4)'$, there have been a few developments (see   \cite{Kim-EquiQH} and \cite{LaSh-QDSP, AnCh, IMN} respectively).
For $QH^*_T(F\ell_{n_1, \cdots, n_k; n+1})$, we expect that our Theorem \ref{thmequiQPR}  leads to an alternative approach to the earlier studies     on $(3)'$ and  $ (4)'$. Indeed, we will illustrate this  for the special case of  complex Grassmannians \cite{Mihalcea-EQGiambelli}, by using Theorem \ref{thmQPRforComplexGrassmannian}.

It is another one of the main results of the present paper that the ``quantum to classical" principle holds among various equivariant Gromov-Witten invariants of  $G/P$. Applying it for   $G=PSL(n+1, \mathbb{C})$,
we achieve the above theorems.
 In    \cite{LeungLi-functorialproperties,Li-functorial}, Leung and the second author discovered
   a functorial relationship between  the  quantum cohomologies of   complete and partial flag varieties, in terms of  filtered algebraic structures on $QH^*(G/B)$. A ``quantum to classical" principle among various  Gromov-Witten invariants of $G/B$ was therefore obtained \cite{LeungLi-QuantumToClassical}  in a combinatorial way. Such a   principle was also shown  for some other cases
  \cite{BKT-GWinv,CMP-I,ChPe} in a geometric way. It has led to nice applications in finding Pieri-type formulas, such as
  \cite{BKT-Isotropic,LeungLi-QPR}.
  In the present paper, we generalize the work  \cite{LeungLi-functorialproperties} to the equivariant quantum cohomology $QH^*_T(G/B)$ in the following case.
 \vspace{0.2cm}

\noindent\textbf{Theorem \ref{thmfiltration}.}  {\itshape Let $P\supset B$ be a parabolic subgroup of $G$ such that $P/B$ is     isomorphic to the complex projective line $\mathbb{P}^1$.
With respect to the natural projection  $G/B\to G/P$,    there is a   $\mathbb{Z}^2$-filtration $\mathcal{F}=\{F_{\mathbf{a}}\}$ on $QH^*_T(G/B)$ respecting the algebra structure: $F_{\mathbf{a}}\star F_{\mathbf{b}}\subset F_{\mathbf{a}+\mathbf{b}}$ for any $\mathbf{a}, \mathbf{b}\in \mathbb{Z}^2$.
 Moreover, the associated
graded algebra $Gr^{\mathcal{F}}(QH^*_T(G/B))$ is isomorphic to  $QH^*(P/B)\otimes_{\mathbb{Z}} QH^*_T(G/P)$ as  $\mathbb{Z}^2$-graded $\mathbb{Z}$-algebras, after localization.}

\vspace{0.2cm}

\noindent Explicit construction of the $\mathbb{Z}^2$-filtration will be given in section 2.2.
 As a consequence, we obtain \textbf{Theorem \ref{thmforQtoC}},  giving identities among  various  equivariant Gromov-Witten invariants of $G/B$. This is
 an   extension of   Theorem 1.1  \cite{LeungLi-QuantumToClassical} in exactly  the same form.
 Together with an equivariant extension  of Peterson-Woodward comparison formula (see Proposition \ref{propcomparison}), we expect that
 Theorem \ref{thmforQtoC} leads to nice applications in the equivariant quantum Schubert calculus  for various  $G/P$.
 Indeed, for  $G=PSL(n+1, \mathbb{C})$,   we  can  reduce all the relevant
 equivariant Gromov-Witten invariants in Theorem \ref{thmequiQPR} to   Pieri-type structure coefficients for  $H^*_T(F\ell_{1, 2, \cdots, n; n+1})$, by
 applying    Theorem \ref{thmforQtoC} repeatedly. Combining such reductions with Robinson's equivariant Pieri rule  \cite{Robinson}, we
   obtain Theorem \ref{thmequiQPR}. It will be very interesting to explore a simpler and conceptually much cleaner  proof  by a kind of inductive argument  
based on    Mihalcea's characterization of the structure coefficients  via the equivariant quantum Chevalley rule \cite{Miha-EQCRandCri}.  
   We can also find nice applications on  the equivariant quantum Pieri rules for orthogonal isotropic  Grassmannians, in a joint work in progress by   the second author and Ravikumar.

   This paper is organized as follows. In section 2, we introduce basic notations, and prove the main technical results for homogeneous varieties $G/P$ of general Lie type.
In  section 3, we show an equivariant quantum Pieri rule for $F\ell_{n_1, \cdots, n_k; n+1}$, and give a simplification in the special case of  complex Grassmannians.
Finally in the appendix, we give  alternative proofs of a ring presentation of $QH^*_T(Gr(m, n+1))$ and the equivariant quantum Giambelli formula for $Gr(m, n+1)$.

\subsection*{Acknowledgements}

The authors thank Leonardo Constantin Mihalcea for his generous helps. The authors also  thank Anders Skovsted Buch, Ionut Ciocan-Fontanine, Thomas Lam, and Naichung Conan Leung     for helpful conversations. The authors would like to thank the anonymous referees for their valuable comments on an earlier version of the manuscript. The second author is supported by IBS-R003-D1.

\section{A $\mathbb{Z}^2$-filtration on $QH^*_T(G/B)$ and its consequences}
In this section, we show a $\mathbb{Z}^2$-filtered algebraic structure on $QH^*_T(G/B)$, with respect to a choice of a simple root. We also obtain    a number of identities among various  equivariant Gromov-Witten invariants of $G/B$.

\subsection{Preliminaries} We fix the notions here, following  \cite{Hump-LieAlg,Hump-AlgGroup,FuPa}.

Let $\mathfrak{g}$ be   a     complex simple Lie algebra of rank $n$,   $\mathfrak{h}$ be a    Cartan subalgebra of $\mathfrak{g}$, and  $\Delta=\{\alpha_1, \cdots, \alpha_n\}\subset \mathfrak{h}^*$ be a basis of simple roots. Let   $R$   denote the root system of $(\mathfrak{g}, \mathfrak{h})$. We have $R=R^+\sqcup (-R^+)$ with
                       $R^+=R\cap \bigoplus_{i=1}^n{\mathbb{Z}_{\geq 0}}\alpha_i$ called the set of positive roots, and have the Cartan decomposition $\mathfrak{g}=\mathfrak{h}\bigoplus \big(\bigoplus_{\gamma\in R}\mathfrak{g}_{\gamma}\big)$. Let $G$ be the (connected) adjoint group of $\mathfrak{g}$, and $B\subset G$ be the Borel subgroup with $\mathfrak{b}:=\mbox{Lie}(B)=\mathfrak{h}\bigoplus(\bigoplus_{\gamma\in R^+}\mathfrak{g}_\gamma\big)$.
  Each subset $\Delta'$    of $\Delta$ gives a root subsystem    $R_{\Delta'}=R^+_{\Delta}\sqcup (-R^+_{\Delta'})$ where     $R_{\Delta'}^{+}=R^{+}\bigcap \big(\bigoplus_{\alpha\in \Delta'}\mathbb{Z} \alpha\big)$, and defines a parabolic subalgebra by $\mathfrak{p}(\Delta'):=\mathfrak{b}\bigoplus\big(\bigoplus_{\gamma\in -R^+_{\Delta'}} \mathfrak{g}_{\gamma}\big)$. This gives rise  to a  one-to-one correspondence between the subsets $\Delta_P$ of $\Delta$ and the
  parabolic subgroups $P\subset G$ that contain $B$.
   In particular, we denote by $P_\beta$ the   parabolic subgroup corresponding to a subset $\{\beta\}\subset \Delta$.   We notice that $P_\beta$ is a minimal subgroup among those parabolic subgroups $P\supsetneq B$, and $P_{\beta}/B$ is isomorphic to the complex projective line $\mathbb{P}^1$.

Let    $\{\alpha_1^\vee, \cdots, \alpha_n^\vee\}\subset\mathfrak{h}$ be the simple coroots,  $\{\chi_1^\vee, \cdots, \chi_n^\vee\}\subset \mathfrak{h}$ be the fundamental coweights,    $\{\chi_1, \cdots, \chi_n\}\subset \mathfrak{h}^*$ be the
       fundamental weights,   and $\rho:=\sum_{i=1}^n\chi_i$.
 Let    $\langle\cdot, \cdot\rangle
                    :\mathfrak{h}^*\times\mathfrak{h}\rightarrow \mathbb{C}$ denote the natural pairing.
        Every simple root $\alpha_i$ labels   a    simple reflection  $s_i:=s_{\alpha_i}$,  which     maps  $\lambda\in\mathfrak{h}$ and $\gamma\in\mathfrak{h}^*$ to
        $s_{i}(\lambda)=\lambda-\langle \alpha_i, \lambda\rangle\alpha_i^\vee$ and $s_{i}(\gamma)=\gamma-\langle \gamma, \alpha_i^\vee\rangle\alpha_i$ respectively.
Let $W$ denote   the Weyl group  generated by  all the
                     simple reflections,  and     $W_{P}$ denote
     the subgroup of $W$ generated by $\{s_\alpha~|~ \alpha \in \Delta_{P}\}$.
Let  $\ell: W\rightarrow \mathbb{Z}_{\geq 0}$ denote   the
standard length function,    $\omega$ (resp. $\omega_P$) denote  the longest element in $W$ (resp. $W_P$),
      and $W^P$ denote the  subset of $W$ that consists of    minimal length  representatives  of the cosets in $W/W_P$.
   Denote $Q^\vee:=\bigoplus_{i=1}^n\mathbb{Z}\alpha_i^\vee$  and  $Q^\vee_P:=\bigoplus_{\alpha_i\in \Delta_P}\mathbb{Z}\alpha_i^\vee$.
 Every $\gamma\in R$ is given by    $\gamma=w(\alpha_i)$ for some $(w, \alpha_i)\in W\times \Delta$, then the coroot $\gamma^\vee:=w(\alpha_i^\vee)\in Q^\vee$  and the reflection
          $s_\gamma:=ws_iw^{-1}\in W$ are both  independent of   the expressions of
           $\gamma$.

Let $T$ be the maximal complex torus of $G$ with $\mathfrak{h}=\mbox{Lie}(T)$, and $N(T)$ denote the normalizer of $T$ in $G$.  There is a canonical isomorphism $W\overset{\cong}{\rightarrow} N(T)/T$ by $w\mapsto \dot w T$. We then have a Bruhat decomposition of
 the   homogeneous variety $G/P$, given by $G/P=\bigsqcup_{w\in W^P} B^-\dot w P/P$, where $B^-$ denotes the opposite Borel subgroup, and each cell $B^-\dot w P/P$ is isomorphic to
  $\mathbb{C}^{\dim_{\mathbb{C}} G/P - \ell(w)}$. As a consequence,
the integral (co)homology of the homogeneous variety $G/P$ has an additive $\mathbb{Z}$-basis of Schubert (co)homology classes $\sigma_w$  (resp. $\sigma^w$) of (co)homology degree
$2\ell(w)$,  indexed by $w\in W^P$. Here $\sigma^w=\mbox{P.D.}([X^w])$ is the Poincar\'e dual of the fundamental class of the Schubert subvariety $X^w:=\overline{B^-\dot w P/P}\subset G/P$, and   $\sigma_w$ is the    fundamental class of the Schubert subvariety $X_w:=\overline{B\dot w P/P}\subset G/P$.

     We consider the integral $T$-equivariant cohomology $H^*_T(G/P)$ with respect to   the natural (left) $T$-action on $G/P$. Every Schubert subvariety $X^w$ is $T$-invariant and of  complex codimension $\ell(w)$, and hence determines an equivariant cohomology class in $H^{2\ell(w)}_T(X)$, which we still denote as $\sigma^w$ by abuse of notations. The equivariant cohomology $H^*_T(G/P)$ is an $H^*_T(\mbox{pt})$-module with an $H^*_T(\mbox{pt})$-basis of the equivariant Schubert classes $\sigma^w$. Here  $H^*_T(\mbox{pt})$, denoting   the $T$-equivariant cohomology of a point equipped with the trivial $T$-action,   is isomorphic to  the symmetric
     algebra of the character group of $T$. Since $G$ is adjoint, we have
      $S:=H^*_T(\mbox{pt})=\mathbb{Z}[\alpha_1, \cdots, \alpha_n]$.

The second integral homology   $H_2(G/P,\mathbb{Z})$ has a basis of   Schubert curve classes $\{\sigma_{s_{\alpha}}\}_{\alpha\in \Delta\setminus\Delta_P}$.
  Therefore, it can be
          identified  with  $Q^\vee/Q^\vee_P$,            by  $\sum\limits_{\alpha_j\in \Delta\setminus\Delta_P}a_j\sigma_{s_{\alpha_j}} \mapsto \lambda_P=\sum\limits_{\alpha_j\in \Delta\setminus\Delta_P}a_j\alpha_j^\vee+Q^\vee_P$. We call $\lambda_P$ effective, if all $a_j$ are nonnegative integers, i.e., if the   associated   function     $q_{\lambda_P}:=\prod\limits_{\alpha_j\in \Delta\setminus\Delta_P}q_{\alpha_j^\vee+Q^\vee_P}^{a_j}$     is a  monomial in the polynomial ring $\mathbb{Z}[\mathbf{q}]$ of
          indeterminate  variables $q_{\alpha_j^\vee+Q^\vee_P}$.
       The integral (\textit{small}) \textit{quantum cohomology ring}   $QH^*(G/P)=(H^*(G/P)\otimes\mathbb{Z}[\mathbf{q}],  \bullet_P)$ of $G/P$
      is a deformation of the ring structure of $H^*(G/P)$. The quantum multiplication is defined by incorporating  genus zero, three-point Gromov-Witten invariants, i.e., intersection numbers on   the moduli spaces of stable maps $\overline{\mathcal{M}}_{0, 3}(G/P, \mathbf{d})$, with respect to three classes pull-back from $H^*(G/P)$ via the natural evaluation maps.
           The moduli space   $\overline{\mathcal{M}}_{0, 3}(G/P, \mathbf{d})$ admits a natural $T$-action induced from the one on the target space $G/P$, and the evaluation maps  are all $T$-equivariant.
        The so-called $T$-equivariant Gromov-Witten invariants  are polynomials in   $S$,    defined   by pulling back classes in $H^*_T(G/P)$ to $H^*_T\big(\overline{\mathcal{M}}_{0, 3}(G/P, \mathbf{d})\big)$ and taking integration over the moduli space with  the equivariant Gysin push forward map \cite{Kim-EquiQH}. The Schubert classes $\sigma^u$ form an $S[\mathbf{q}]$-basis of the commutative \textit{$T$-equivariant quantum cohomology ring} $QH^*_T(G/P)$.
       The   structure coefficients $N_{u, v}^{w, \lambda_P}$ in the equivariant quantum product, $$\sigma^u\star_P \sigma^v =\sum_{w\in W^P, \lambda_P\in Q^\vee/Q^\vee_P}    N_{u,v}^{w, \lambda_P}\sigma^wq_{\lambda_P},$$
     are homogenous polynomials in $S$. The classical limit $N_{u, v}^{w, \mathbf{0}}$ coincides with the coefficient of $\sigma^{w}$ in the equivariant  product $\sigma^u\circ \sigma^v$ in $H_T^*(G/P)$.
     The non-equivariant limit $N_{u, v}^{w, \lambda_P}\big|_{\alpha_1=\cdots=\alpha_n=0}$ is a  Gromov-Witten invariant, coinciding with
      the coefficient
        of $\sigma^wq_{\lambda_P}$ in the quantum product $\sigma^u\bullet_P\sigma^v$ in $QH^*(G/P)$.

   There is an    equivariant quantum  Chevalley formula stated by Peterson  \cite{Peterson}  and   proved by Mihalcea  \cite{Miha-EQCRandCri}, which concerns     the multiplication by    Schubert divisor classes in $QH^*_T(G/P)$.  We review    the special case of it when $P=B$ as follows. In this case, we notice that
    $Q^\vee_B=0$, $W_B=\{1\}$ and $W^B=W$. Hence  we will simply denote %$\star:=\star_B$,
       $\lambda:=\lambda_B$ and $q_j:=q_{\alpha_j^\vee}$,  whenever there is no confusion.

 \begin{prop}[Equivariant quantum Chevalley formula for $G/B$]\label{propQChevalley}
 For   any simple reflection $s_i$ and any $u$ in $W$,    in $QH_T^*(G/B)$, we have
    $$\sigma^{s_i}\star \sigma^u=(\chi_i-u(\chi_i))\sigma^u+ \sum \langle \chi_i, \gamma^\vee \rangle \sigma^{us_\gamma}+
             \sum \langle\chi_i, \gamma^\vee \rangle q_{\gamma^\vee}\sigma^{us_\gamma},
              $$
 the first summation over those $\gamma\in R^+$ satisfying $\ell(us_\gamma)=\ell(u)+1$, and the second summation over those $\gamma\in R^+$ satisfying  $\ell(us_\gamma)=\ell(u)+1-\langle2\rho, \gamma^\vee\rangle$.

\end{prop}

   Despite of the lack of   geometric meaning, the structure coefficients   $N_{u, v}^{w, \lambda_P}$ for $QH^*_T(G/P)$   enjoy a positivity property \cite{Miha-positivity}. Here is a precise statement for  $P=B$.
\begin{prop}[Positivity]\label{propPositivity}
 Let $u, v, w\in W$, $\lambda\in Q^\vee$, and $d:=\ell(u)+\ell(v)-\ell(w)-\langle 2\rho, \lambda\rangle$.  The structure coefficient
     $N_{u, v}^{w, \lambda}$ for  $QH^*_T(G/B)$ is a homogeneous polynomial of degree $d$ in $\mathbb{Z}_{\geq 0}[\alpha_1, \cdots, \alpha_n]$, provided that
     $\lambda$ is effective and $d\geq 0$, and zero otherwise.
\end{prop}
\noindent We remark that the structure coefficients for equivariant (quantum) product of the equivariant (quantum) Schubert classes determined by the $T$-invariant Schubert varieties $X_{\omega_0w}$ enjoy the Graham-positivity \cite{Grah, Miha-positivity}, i.e., they take values in $\mathbb{Z}_{\geq 0}[-\alpha_1, \cdots, -\alpha_n]$ instead.

\subsection{Main results}\label{subsetmainthms}

Let $\beta\in \Delta$. The natural projection
 $G/B\rightarrow G/P_\beta$ is a  bundle  with fiber  $P_\beta/B\cong \mathbb{P}^1$.
 As in \cite{LeungLi-QuantumToClassical}, we define a map $\mbox{sgn}_\beta:  W\rightarrow \{0, 1\}$  by $\mbox{sgn}_\beta(w)=1$ if $w(\beta)\in-R^+$, or 0 otherwise. In other words, we have
         $$\mbox{sgn}_\beta(w)=\begin{cases}
             1, & \mbox{if } \ell(w)-\ell(ws_\beta)>0\\
             0, &\mbox{if } \ell(w)-\ell(ws_\beta)\leq 0
         \end{cases}.
$$

For $I=(i_1, \cdots, i_n)\in \mathbb{Z}^n$, we denote   $|I|:=i_1+\cdots+i_n$ and $\alpha^I:=\alpha_1^{i_1}\cdots \alpha_n^{i_n}$

\begin{defn}\label{defgradingequiv} With respect to $\beta\in \Delta$, we define a map          $gr_\beta:  W\times \mathbb{Z}^n \times Q^\vee\longrightarrow \mathbb{Z}^2,$
 {\upshape     \begin{align*}
%         gr_\beta: & W\times Q^\vee\times \mathbb{Z}^n\longrightarrow \mathbb{Z}^2;\\
                  &\,  gr_\beta(w, I,  \lambda):=(\mbox{sgn}_\beta(w)+\langle \beta, \lambda\rangle, \ell(w)+|I|+\langle 2\rho, \lambda\rangle-\mbox{sgn}_\beta(w)-\langle \beta, \lambda\rangle).
     \end{align*}
}
\end{defn}

 The equivariant quantum cohomology ring $QH^*_T(G/B)$ admits a  $\mathbb{Z}$-basis  $\sigma^{w}\alpha^Iq_\lambda$,
 with $w\in W$  and $\alpha^Iq_\lambda\in \mathbb{Z}[\mathbf{\alpha}, \mathbf{q}]$.
Naturally, we define the  grading of  $\sigma^w\alpha^Iq_\lambda$ to be $gr_\beta (w, I,  \lambda)$.
Therefore,    we obtain a     family   % $\mathcal{F}=\{F_{\mathbf{a}}\}_{\mathbf{a}\in \mathbb{Z}^2}$
           of $\mathbb{Z}$-vector subspaces of $QH^*_T(G/B)$ by

     $$  \mathcal{F}:=\{F_{\mathbf{a}}\}_{\mathbf{a}\in \mathbb{Z}^2}\quad \mbox{ with }\quad F_{\mathbf{a}}:=\bigoplus\limits_{gr_\beta({w},I,  \lambda)\leq \mathbf{a}}\mathbb{Z}\sigma^w\alpha^Iq_\lambda\subset QH^*_T(G/B).$$
  Here we are considering  the \textit{lexicographical order} on $\mathbb{Z}^2$. That is, $(a_1, a_2)<(b_1, b_2)$ if and only if either $a_1<b_1$ or $(a_1=b_1$ and $a_2<b_2)$.
  The   associated $\mathbb{Z}^2$-graded vector space
  with respect to $\mathcal{F}$ is then given by
   $$Gr^{\mathcal{F}}(QH^*_T(G/B))=\bigoplus_{\mathbf{a}\in \mathbb{Z}^2} Gr_\mathbf{a}^{\mathcal{F}} \quad \mbox{ where } \quad Gr_{\mathbf{a}}^{\mathcal{F}}:=F_{\mathbf{a}}\big/\cup_{\mathbf{b}<\mathbf{a}}F_{\mathbf{b}}.$$

\begin{lem-defn}[Lemma 1 of \cite{Woodward}]\label{lemmalifting}
     Let $\lambda_P\in Q^\vee/Q_P^\vee$. Then there is a unique $\lambda_B\in Q^\vee$ such that $\lambda_P=\lambda_B+Q_P^\vee$ and
                $\langle \gamma, \lambda_B\rangle  \in \{0, -1\}$ for all $\gamma\in R^+_P$.
                We call $\lambda_B$    the  \textbf{Peterson-Woodward lifting} of $\lambda_P$.
\end{lem-defn}
\noindent Thanks to the above lemma, we obtain an  injective morphism   of
${S}$-modules   $$\psi_{\Delta, \Delta_P}: QH^*_T(G/P)\longrightarrow QH^*_T(G/B) $$
       {defined by} $ \sigma_P^wq_{\lambda_P}\mapsto \sigma_B^{w\omega_{P}\omega_{P'}}q_{\lambda_B} $. Here $\omega_{P'}$ denotes the longest element in the Weyl subgroup generated by the simple reflections $\{s_\alpha~|~ \alpha\in \Delta_P, \langle \alpha, \lambda_B\rangle=0\}$, and the subscript ``$P$" in the Schubert classes $\sigma^w_P$ for $G/P$ is used in order to distinguish them from those Schubert classes for $G/B$. In the special case when  $P=P_{\beta}$, we simply denote
        $\psi_\beta:=\psi_{\Delta, \{\beta\}}.$

Our first main result is the next theorem, giving  an equivariant    generalization of  the special case of Theorems 1.2 and 1.4 of \cite{LeungLi-functorialproperties} when $P=P_\beta$.
We take an isomorphism   $QH^*(\mathbb{P}^1)\cong {\mathbb{Z}[x, t]\over \langle x^{2}-t\rangle }$ of algebras.
\begin{thm}\label{thmfiltration}
 The filtration   $\mathcal{F}$ gives a   $\mathbb{Z}^2$-filtered algebraic structure on    $QH^*_T(G/B)$.
     That is, we have $F_\mathbf{a}\star F_\mathbf{b}\subset F_{\mathbf{a}+\mathbf{b}}$ for any $\mathbf{a}, \mathbf{b}\in \mathbb{Z}^2$.
%Furthermore,
            {\upshape $$\begin{array}{cr}
   \mbox{{\itshape The map}}\quad   \Psi_{\scriptsize\mbox{ver}}^\beta:&  QH^*(\mathbb{P}^1) \longrightarrow  Gr_{\scriptsize\mbox{ver}}^{\mathcal{F}}:=\bigoplus\limits_{i\in \mathbb{Z}}Gr_{(i, 0)}^{\mathcal{F}}\subset Gr^{\mathcal{F}}(QH^*_T(G/B)),\quad{}
       \end{array}$$}
defined by $x\mapsto \overline{\sigma^{s_\beta}}$ and $t\mapsto \overline{q_{\beta^\vee}}$, is an isomorphism  of   $\mathbb{Z}$-algebras.
     {\upshape $$  \mbox{{\itshape The map}}\quad  \Psi_{\scriptsize\mbox{hor}}^\beta:   QH^*_T(G/P_\beta) \longrightarrow   Gr_{\scriptsize\mbox{hor}}^{\mathcal{F}}:=\bigoplus\limits_{j\in \mathbb{Z}}Gr_{(0, j)}^{\mathcal{F}}\subset Gr^{\mathcal{F}}(QH^*_T(G/B)),$$}
defined by  $\sigma^w\alpha^Iq_{\lambda_{P_\beta}} \mapsto  \overline{\psi_{\beta}(\sigma^w\alpha^Iq_{\lambda_{P_\beta}})}$,
is an  isomorphism  of   $S$-algebras.
\end{thm}

 \begin{remark}\label{corgradedisom}
    There is  a  $\mathbb{Z}^2$-filtration   $\mathcal{F}'$ on    $QH^*_T(G/B)[q_{\beta^\vee}^{-1}]$, naturally extended from $\mathcal{F}$.
      The above {\upshape $\Psi_{\scriptsize\mbox{ver}}^\beta$}, {\upshape $\Psi_{\scriptsize\mbox{hor}}^\beta$} induce an isomorphism of $\mathbb{Z}^2$-graded $\mathbb{Z}$-algebras:
        {\upshape $$ \Psi_{\scriptsize\mbox{ver}}^\beta\otimes \Psi_{\scriptsize\mbox{hor}}^\beta :\, \, Gr^{\mathcal{F}'}\big(QH^*_T(G/B)[q_{\beta^\vee}^{-1}]\big)
                      \overset{\cong}{\longrightarrow} QH^*(\mathbb{P}^1)[t^{-1}]\otimes_\mathbb{Z}QH^*_T(G/P).$$}
\end{remark}

\bigskip

Our second main result is the next   generalization  of    \cite[Theorem 1.1]{LeungLi-QuantumToClassical} to   $QH^*_T(G/B)$, with  the statements  exactly of the same form. We simply denote $\mbox{sgn}_i:=\mbox{sgn}_{\alpha_i}$.

\begin{thm}\label{thmforQtoC}
 Let $u, v, w\in W$ and $\lambda\in Q^\vee$. The coefficient   $N_{u, v}^{w, \lambda}$ of $\sigma^wq_\lambda$ in the  equivariant quantum product $\sigma^u\star \sigma^v$ in $QH^*_T(G/B)$ satisfies the following.
   \begin{enumerate}
      \item   $N_{u, v}^{w, \lambda}=0$  unless  {\upshape $\mbox{sgn}_i(w)+\langle \alpha_i, \lambda\rangle \leq \mbox{sgn}_i(u)+\mbox{sgn}_i(v)$} for all $1\leq i\leq n.$
            %%%  $N_{u, v}^{w, \lambda}=0 \mbox{ unless } \mbox{sgn}_\alpha(w)+\langle \alpha, \lambda\rangle \leq \mbox{sgn}_\alpha(u)+\mbox{sgn}_\alpha(v) \mbox{ for all }  \alpha\in \Delta.$
     \item  If    {\upshape $ \mbox{sgn}_k(w)+\langle \alpha_k, \lambda\rangle =\mbox{sgn}_k(u)+\mbox{sgn}_k(v)=2$} for some  $1\leq k\leq n$, then
           {\upshape  $$N_{u, v}^{w, \lambda}=N_{us_k, vs_k}^{w, \lambda-\alpha^\vee_k}=
                 \begin{cases} N_{u, vs_k}^{ws_k, \lambda-\alpha^\vee_k}, &   i\!f    \mbox{ sgn}_k(w)=0 \\
                               \vspace{-0.3cm}   & \\
                         N_{u, vs_k}^{ws_k, \lambda}, & i\!f   \mbox{ sgn}_k(w)=1 \,\, {}_{\displaystyle .}  \end{cases}$$
            }
   \end{enumerate}

\end{thm}

\begin{cor}\label{correduction}
     Let  $u, v, w\in W$, $\alpha\in\Delta$ and $\lambda\in Q^\vee$.
     \begin{enumerate}
        \item If    {\upshape $\langle \alpha, \lambda\rangle\!=\mbox{sgn}_\alpha(u)=0$}\! and {\upshape  $\mbox{sgn}_\alpha(w)=\mbox{sgn}_\alpha(v)=1$},
           then  $N_{u, v}^{w, \lambda}=N_{u, vs_\alpha}^{ws_\alpha, \lambda}.$
         \item If    {\upshape $\langle \alpha, \lambda\rangle\!=\mbox{sgn}_\alpha(u)=1$} \!and {\upshape  $\mbox{sgn}_\alpha(w)=\mbox{sgn}_\alpha(v)=0$},
           then  $N_{u, v}^{w, \lambda}=N_{us_\alpha, v}^{ws_\alpha, \lambda-\alpha^\vee}\!.$
      \end{enumerate}
 \end{cor}

\begin{proof}
 % When the hypothesis of (1) holds,
  For part (1), we note
      $\langle \alpha, \lambda+\alpha^\vee\rangle=2, \mbox{sgn}_\alpha(ws_{\alpha})=0,  \mbox{sgn}_\alpha(us_\alpha)= \mbox{sgn}_\alpha(v)=1$.
     Applying ``$(u, v, w, \lambda, \alpha_k)$" in Theorem \ref{thmforQtoC} (2) to $(v, us_\alpha,  ws_{\alpha}, \lambda+\alpha^\vee, \alpha)$, we have  $N_{v, us_{\alpha}}^{ws_\alpha, \lambda+\alpha^\vee}=N_{vs_\alpha, us_\alpha s_{\alpha}}^{ws_{\alpha}, \lambda+\alpha^\vee-\alpha^\vee}=
                  N_{v, us_\alpha s_\alpha}^{ws_\alpha s_\alpha, \lambda+\alpha^\vee-\alpha^\vee}
                              $. That is, $N_{vs_\alpha, u}^{ws_{\alpha}, \lambda}=
                  N_{v, u}^{w, \lambda}$. %Hence, statement   (1) follows.
       Similarly, we conclude     (2),  by applying
     ``$(u, v, w, \lambda, \alpha_k)$"   to $(u, vs_\alpha,  ws_{\alpha}, \lambda, \alpha)$.
\end{proof}
%%%\begin{lemma}\label{lemmadescentcycling}
%%%   Let  $u, v, w\in W$, $\alpha\in\Delta$ and $\lambda\in Q^\vee$. If $\ell(us_\alpha)>\ell(u)$, $\ell(vs_\alpha)<\ell(v)$,  $\ell(ws_\alpha)<\ell(w)$ and %%%$\langle \alpha, \lambda\rangle=0$, then %among the structure constants for $QH^*_T(G/B)$,
%%%     $N_{u, v}^{w, \lambda}=N_{u, vs_\alpha}^{ws_\alpha, \lambda}$.
%%%\end{lemma}
\begin{remark}
  When $\lambda=0$,   Corollary \ref{correduction} {  (1)}   was also known as the ``descent-cycling" condition for $H^*_T(G/B)$ in    \cite{knutson:noncomplex}.
 \end{remark}

\subsection{Equivariant Peterson-Woodward comparison formula}
There is a comparison formula,  originally stated by  Peterson  \cite{Peterson} and proved by Woodward \cite{Woodward}. It tells   that every genus zero, three-point Gromov-Witten invariant  of $G/P$ coincides with a corresponding Gromov-Witten invariant of    $G/B$.
It has played  an important role in  the earlier works \cite{LeungLi-functorialproperties, LeungLi-QuantumToClassical, Li-functorial}. In order to prove our main results, we need the  equivariant extension  of the   comparison formula as follows.  Our readers may skip this subsection first, by assuming the following proposition.
  \begin{prop}[Equivariant Peterson-Woodward comparison formula]\label{propcomparison}
             For any $u, v, w\in W^P$ and $\lambda_P\in Q^\vee/Q^\vee_P$, we have
                    $$N_{u,v}^{w, \lambda_P }=N_{u, v}^{w\omega_P\omega_{P'},   \lambda_B},$$
              % Here  $\omega_{P}$ (resp.
              where  $\lambda_B$ denotes the Peterson-Woodward lifting of $\lambda_P$, and
               $\omega_{P'}$  denotes the longest element in the Weyl subgroup generated by
               $\{s_\alpha~|~ \alpha  \in \Delta_{P}, \langle  \alpha, \lambda_B\rangle =0\}$.
 \end{prop}
\noindent  That is,  the coefficient    of $\sigma^w_Pq_{\lambda_P}$ in the equivariant quantum product   $\sigma^u_P\star_P\sigma^v_P$ in $QH^*_T(G/P)$
              coincides with the coefficient of $\sigma_B^{w\omega_P\omega_{P'}}q_{\lambda_B}$  in
                   $\sigma^u_B\star_B\sigma^v_B$ in $QH^*_T(G/B)$.
\begin{remark}
The above statement is exactly of the same as
   Theorem 10.15 (2) of \cite{LaSh-GoverPaffineGr}, which is an equivalent version of the
 non-equivariant Peterson-Woodward comparison formula in terms of Gromov-Witten invariants in \cite{Woodward}.
\end{remark}
\noindent The geometric method   in \cite{Woodward} might also be valid in  the equivariant setting, while a rigorous argument is missing.   In this subsection, we will devote to a proof of Proposition \ref{propcomparison}, by a direct translation of Corollary 10.22 of \cite{LaSh-GoverPaffineGr} by Lam and Shimozono.
   We will have to introduce some notations on the affine Kac-Moody algebras, which, however, will be used in the rest of this subsection only.

The affine Weyl group of $G$ is the semi-direct product  $W_{\scriptsize\mbox{af}}:=W\ltimes Q^\vee$, in which the image of $\lambda\in Q^\vee$ in  $W_{\scriptsize\mbox{af}}$ is a translation, denoted as $t_\lambda$. We have $t_{w(\lambda)}=wt_\lambda w^{-1}$ and $t_{\lambda}t_{\lambda'}=t_{\lambda+\lambda'}$ for all $w\in W$ and $\lambda, \lambda'\in Q^\vee$. Let $\tilde Q^\vee$ denote the set of anti-dominant elements in $Q^\vee$, i.e.,
$\tilde Q^\vee=\{\lambda\in Q^\vee~|~ \langle \alpha, \lambda\rangle \leq 0 \mbox{ for all } \alpha\in \Delta\}$.
Denote $W_{\scriptsize\mbox{af}}^-:=\{wt_\lambda~|~ \lambda\in \tilde Q^\vee, \mbox{ and if } \alpha\in \Delta \mbox{ satisfies } \langle \alpha, \lambda\rangle=0 \mbox{ then } w(\alpha)\in R^+\}$, which consists of the minimal length representatives of the cosets $W_{\scriptsize\mbox{af}}/W$. The (level zero) action of $W_{\scriptsize\mbox{af}}$  on the affine root system  $ R_{\scriptsize\mbox{af}}= R_{\scriptsize\mbox{af}}^+\bigsqcup (-R_{\scriptsize\mbox{af}}^+)$ is given by $wt_\lambda(\gamma+m\delta)=w(\gamma)+(m-\langle \gamma, \lambda\rangle)\delta$, in which $\delta$ denotes the null root and
   $ R_{\scriptsize\mbox{af}}^+:=\{\gamma+m\delta~|~ m\in \mathbb{Z}^+ \mbox{ or } (m=0 \mbox{ and  } \gamma\in R^+)\}$.
Denote the subgroup           $ (W_P)_{\scriptsize\mbox{af}}:=W_P\ltimes Q^\vee_P$ and the subset
 $ (W^P)_{\scriptsize\mbox{af}}:=\{x\in  W_{\scriptsize\mbox{af}}~|~ x(\gamma+m\delta)\in R_{\scriptsize\mbox{af}}^+ \mbox{ for all }
             \gamma+m\delta\in R_{\scriptsize\mbox{af}}^+ \mbox{ with } \gamma\in R_P\}$. %Then we can define a map
            %% $\phi_P: W_{\scriptsize\mbox{af}}\rightarrow (W^P)_{\scriptsize\mbox{af}}$ by $x\mapsto x_1$ in the notation of the next lemma. (Such a map is denoted as ``$\pi_P$" in \cite{LaSh-GoverPaffineGr}.)

         \begin{lemma}[See e.g. Lemma 10.6 and Proposition 10.10 of \cite{LaSh-GoverPaffineGr}]\label{lemmapropfactorization} For every {\upshape $x\in W_{\scriptsize\mbox{af}}$}, there is a unique factorization
                 $x=x_1x_2$ with {\upshape $x_1\in (W^P)_{\scriptsize\mbox{af}}$} and {\upshape $x_2\in (W_P)_{\scriptsize\mbox{af}}$}.
This defines a map\footnote{The   map is denoted as  $\pi_P$  in \cite{LaSh-GoverPaffineGr}.}
             {\upshape $\phi_P: W_{\scriptsize\mbox{af}}\rightarrow (W^P)_{\scriptsize\mbox{af}}$} by $x\mapsto x_1$.   Then for any $w\in W^P$, we have $\phi_P(wx)=w\phi_P(x)$.
         \end{lemma}

 \begin{prop}[Corollaries 9.3 and 10.22    of \cite{LaSh-GoverPaffineGr}]\label{propcomparisoninH(affine)}
    Let $u, v, w\in W^P$ and $\lambda_P\in Q^\vee/Q^\vee_P$. Pick $\mu, \kappa, \eta\in \tilde Q^\vee$ such that $u\phi_P(t_{\mu})$,
     $v\phi_P(t_{\kappa})$ and {\upshape $w\phi_P(t_{\eta})\in  W_{\scriptsize\mbox{af}}^-\cap (W^P)_{\scriptsize\mbox{af}}$} where $\lambda_P=\eta-\mu-\kappa+Q^\vee_P$. Write $u\phi_P(t_{\mu})=u't_{\mu'}, v\phi_P(t_{\kappa})=v't_{\kappa'}$ and $w\phi_P(t_{\eta})=w't_{\eta'}$, where $u', v', w'\in W$ and $\mu', \kappa', \eta'\in Q^\vee$.  Then we have
       $$N_{u, v}^{w, \lambda_P}= N_{u', v'}^{w', \eta'-\mu'-\kappa'},$$
    in which the left-hand  (resp. right-hand) side is a structure coefficient  of the equivariant quantum product
      $\sigma^u_P\star_P\sigma^v_P \in QH^*_T(G/P)$ (resp. $\sigma^{u'}_B\star_B\sigma^{v'}_B\in  QH^*_T(G/B)$).
 \end{prop}

\bigskip

\begin{proof}[Proof of Proposition \ref{propcomparison}]
Let $\mu=-12(n+1)M\sum\nolimits_{\alpha\in \Delta\setminus \Delta_P}\chi_\alpha^\vee$ and $\eta=2\mu+\lambda_B$,  where    $M:=\max\{|\langle \alpha, \lambda_B\rangle|+1 ~|~ \alpha\in\Delta\}$. Then $\mu,  \eta$ are both in $Q^\vee$ (since the determinant of the Cartan matrix $\big(\langle \alpha_i, \alpha_j^\vee\rangle\big)$ is equal to $1, 2, 3, 4$ or $n+1$).

 Clearly, for $\alpha\in\Delta\setminus\Delta_P$, we have $\langle\alpha, \mu\rangle<0$ and $\langle\alpha, \eta\rangle<0$. For $\gamma\in R_P^+$ (in particular for $\gamma\in \Delta_P$), we have  $\langle\gamma, \mu\rangle=0$ and $\langle\gamma, \eta\rangle=\langle\gamma, \lambda_B\rangle\in\{0, -1\}$. Hence, $\mu, \eta$ are both in $\tilde Q^\vee$, and $ut_\mu,  w\omega_P\omega_{P'}t_{\eta}$ are both in $W_{\scriptsize\mbox{af}}^-$ by noting $u\in W^P$ and $w\omega_P\omega_{P'}\in W^{P'}$.

Let $\gamma+m\delta \in R_{\scriptsize\mbox{af}}^+$ with $\gamma\in R_P$. Since    $t_\mu(\gamma+m\delta)=\gamma+(m-\langle \gamma, \mu\rangle )\delta=
   \gamma+m\delta\in R_{\scriptsize\mbox{af}}^+$,  $t_\mu$ is in $(W^P)_{\scriptsize\mbox{af}}$.    Note $\omega_P\omega_{P'}t_{\eta}(\gamma+m\delta)=\omega_P\omega_{P'}(\gamma)+(m-\langle \gamma, \lambda_B\rangle )\delta$. Clearly,  $\langle \gamma, \lambda_B\rangle$ is in $\{0, 1, -1\}$, and it vanishes if and only if $\gamma\in R_{P'}$. Note    $\omega_P\omega_{P'}(R_{P'}^+)\subset R^+$  and  $\omega_P\omega_{P'}(R_P^+\setminus R_{P'}^+)\subset -R^+$.  As a consequence, we have
    \begin{enumerate}
      \item[ i)]  if $m>2$ or ($m=1$ and $\gamma\in R^+_P$), then    $m-\langle \gamma, \lambda_B\rangle  >0$;
       \item[ ii)]  if $m=1$ and $\gamma\in -R_{P'}^+$, then    $m-\langle \gamma, \lambda_B\rangle  =1>0$;
       \item[iii)]  if $m=1$ and $\gamma\in (-R^+_P)\setminus (-R^+_{P'})$, then $m-\langle \gamma, \lambda_B\rangle  =0$, and we note $\omega_P\omega_{P'}(\gamma)\in R^+$ in this case;
       \item [iv)] if $m=0$, then   $\gamma\in R^+_P$. Further, if $\gamma\in  R^+_P\setminus R^+_{P'}$, then    $m-\langle \gamma, \lambda_B\rangle  =1>0$; if $\gamma\in  R^+_{P'}$, then  $m-\langle \gamma, \lambda_B\rangle  =0$, and we note  $\omega_P\omega_{P'}(\gamma)\in R^+$.
   \end{enumerate}
    Hence,
     $\omega_P\omega_{P'}t_{\eta}$ is also in $(W^P)_{\scriptsize\mbox{af}}$.   Since
      $\omega_P\omega_{P'}\in W_P$,   $\eta-\omega_P\omega_{P'}(\eta)\in Q^\vee_P$. Thus we have the factorizations
      $t_\mu=t_\mu\cdot \mbox{id}$ and $t_{\eta}= (\omega_P\omega_{P'}t_{\eta})\cdot \big((\omega_P\omega_{P'})^{-1}t_{\eta-\omega_P\omega_{P'}(\eta)}\big)$.
     Hence, we have $\phi_P(t_\mu)=t_\mu$ and $\phi_P(t_\eta)=\omega_P\omega_{P'}t_\eta$ due to the uniqueness of the factorization. Hence,
        $ut_\mu,   w\omega_P\omega_{P'}t_{\eta}$ are  in  $(W^P)_{\scriptsize\mbox{af}}$, by noting $u, w\in W^P$ and using Lemma \ref{lemmapropfactorization}.

Now we set $\kappa:=\mu$ and use the same notation as in Proposition \ref{propcomparisoninH(affine)}. It follows immediately from the above arguments that
  $\mu, \kappa, \eta$ satisfy all the hypotheses of Proposition \ref{propcomparisoninH(affine)}, for which we have
         $u'=u, v'=v, w'=w\omega_P\omega_{P'}$, $\mu'=\mu$, $\kappa'=\kappa$, $\eta'=\eta$ and $\eta'-\mu'-\kappa'=\lambda_B$. Therefore the statement follows.
\end{proof}

\subsection{Proof of theorems}\label{subsecproofofthms}
The proofs of the theorems in section \ref{subsetmainthms} are    similar to the corresponding ones in the non-equivariant case in \cite{LeungLi-functorialproperties, LeungLi-QuantumToClassical}.

\subsubsection{Preliminary propositions} We will need the next combinatorial fact.
 \begin{lemma}[Lemmas 3.8 and   3.9 of \cite{LeungLi-functorialproperties}]\label{lemmaUSgamma} Let $u\in W$ and $\gamma\in R^+$ satisfy   $\ell(us_\gamma)=\ell(u)+1-\langle 2\rho, \gamma^\vee\rangle$. If    $\langle \alpha, \gamma^\vee\rangle>0$ for some  $\alpha\in \Delta$, then
     $\ell(u)-\ell(us_\alpha)=\ell(us_\gamma s_\alpha)-\ell(us_\gamma)=1$. Furthermore if $\gamma\neq \alpha$, then $\langle \alpha, \gamma^\vee\rangle=1$.
  \end{lemma}

The next proposition is the generalization of (a special case of) the Key Lemma and Proposition 3.23 of \cite{LeungLi-functorialproperties} to the equivariant quantum cohomology $QH^*_T(G/B)$. %%Note that $\beta\in \Delta$ is an \textit{arbitrarily} fixed simple root.
We would like to remind our readers that a simple root $\beta$ has been fixed in prior.

\begin{prop}\label{propforreduction}
For any  $   1 \leq i\leq n$ and $u\in W$, we have   $F_{\mathbf{a}}\star F_\mathbf{b}\subset F_{\mathbf{a}+\mathbf{b}}$, where
  $\mathbf{a}:=gr_\beta(s_i, \mathbf{0},\mathbf{0})$ and $\mathbf{b}:= gr_\beta(u, \mathbf{0}, \mathbf{0})$.
            Furthermore, we have
         $\overline{\sigma^{s_i}}\star \overline{\sigma^{u}}=\overline{\sigma^{us_i}}$ in $Gr^{\mathcal{F}}(QH^*_T(G/B))$, if
  the following hypotheses $(\diamond)$ hold:   $s_i=s_\beta$ and $u\in W^{P_\beta}$\quad $(\diamond)$.
\end{prop}

  \begin{proof}
  Write $\sigma^{s_i}\star \sigma^{u}=\sum c_{w, I,  \lambda}\sigma^w\alpha^Iq_\lambda$, and denote $\mathbf{d}:=gr_\beta(w, I, \lambda)$.
  The statement to prove is equivalent to the following:
  \begin{enumerate}
    \item[i)]  $\mathbf{d}\leq \mathbf{a}+\mathbf{b}$ whenever the coefficient $c_{w, I, \lambda}$ does not vanish;
    \item[ii)] under the additional hypotheses $(\diamond)$, $\mathbf{d}=\mathbf{a}+\mathbf{b}$ if and only if  $\sigma^w\alpha^Iq_\lambda=\sigma^{us_i}$.
  \end{enumerate}

   Note    $c_{w, I,  \lambda}\neq 0$ only if
             $|\mathbf{d}|= \ell(w)+|I|+\langle 2\rho, \lambda\rangle=1+\ell(u)=|\mathbf{a}|+|\mathbf{b}|$.
    Thus  for nonzero  $c_{w, I,  \lambda}$, $\mathbf{d} $ is less than  (resp. equal to) $\mathbf{a}+\mathbf{b}$   if and only if   $d_1$ is less than (reps. equal to) $a_1+b_1$, where  $\mathbf{a}=(a_1, a_2)$, $\mathbf{b}=(b_1, b_2)$ and $\mathbf{d}=(d_1, d_2)$.
  Note  $a_1+b_1=\mbox{sgn}_\beta(s_i)+\mbox{sgn}_\beta(u)$ and $d_1=\mbox{sgn}_\beta(w)+\langle \beta, \lambda\rangle$.
   Due to Proposition \ref{propQChevalley},  if   $c_{w, I,  \lambda}\neq 0$, then  one of the following cases must hold.
   \begin{enumerate}
       \item $\sigma^w\alpha^Iq_\lambda =\sigma^u\alpha_j$ for some $1\leq j\leq n$, which  must come from  $(\chi_i-u(\chi_i))\sigma^u$. Clearly,
              $d_1=\mbox{sgn}_\beta(u)=b_1\leq a_1+b_1$, and `$<$" holds if we assume $(\diamond)$.
       \item $\sigma^w\alpha^Iq_\lambda=\sigma^{us_\gamma}$ with $\ell(us_\gamma)=\ell(u)+1$. If either of $a_1, b_1$ is nonzero, then $d_1\leq 1\leq a_1+b_1$. Otherwise,  we have $a_1=b_1=0$, i.e., $s_i, u\in W^{P_\beta}$. Due to the canonical injective morphism $H^*(G/P_\beta)\hookrightarrow H^*(G/B)$, $\sigma^{us_\gamma}$ occurs in
             $\sigma^u\cup\sigma^v\in H^*(G/B)\subset QH^*_T(G/B)$ only if $us_\gamma$ lies in $W^{P_\beta}$ as well, i.e., $d_1=0$. Thus $d_1\leq a_1+b_1$. Furthermore we assume $(\diamond)$, then $d_1=a_1+b_1$ only if $us_\gamma=vs_i$ with $\ell(v)=\ell(u)$;
               $c_{vs_i, \mathbf{0},\mathbf{0}}\neq 0$ implies that $u\leq vs_i$ with respect to the Bruhat order, i.e., $u$ is obtained by deleting a simple reflection from a reduced expression of $vs_i$, which implies $u=v$. Thus if both $(\diamond)$ and $d_1=a_1+b_1$ hold, then $us_\gamma=us_i$ and   $c_{us_i, \mathbf{0},\mathbf{0}}=\langle \chi_i, \alpha_i^\vee\rangle=1$.

       \item $\sigma^w\alpha^Iq_\lambda=\sigma^{us_\gamma}q_{\gamma^\vee}$ with $\ell(us_\gamma)=\ell(u)+1-\langle 2\rho, \gamma^\vee\rangle$.
       Furthermore, we have  $a_1=1$, assuming   $(\diamond)$.
           \begin{enumerate}
             \item If $\langle \beta, \gamma^\vee\rangle<0$, then $d_1\leq 0\leq a_1+b_1$, and ``$<$" holds if we assume $(\diamond)$.
             %under the assumption of $(\diamond)$ (as   $a_1=1$).
             \item  If $\langle \beta, \gamma^\vee\rangle=0$, then $us_\gamma(\beta)=u(\beta)$, which  implies $d_1=\mbox{sgn}_\beta(us_\gamma)=\mbox{sgn}_\beta(u)=b_1 \leq a_1+b_1$, and ``$<$" holds  if we assume  $(\diamond)$.
             \item If   $\langle \beta, \gamma^\vee\rangle>0$, then
                            $\mbox{sgn}_\beta(u)=1$ and $\mbox{sgn}_\beta(us_\gamma)=0$ by Lemma \ref{lemmaUSgamma}. If
                             $\gamma\neq \beta$, then   $d_1=\langle \beta, \gamma^\vee\rangle=1=b_1\leq a_1+b_1$, and "$<$" holds if we assume $(\diamond)$. If $\gamma=\beta$, then $\langle \chi_i, \gamma^\vee\rangle\neq 0$ implies that $\alpha_i=\beta$, and consequently
                              $d_1=2=1+1=a_1+b_1$. Since  $\mbox{sgn}_\beta(u)=1$, $u\not\in W^{P_\beta}$. Hence,
                           the hypotheses $(\diamond)$ cannot hold  in this case.
           \end{enumerate}
   \end{enumerate}
          Hence, the statement follows.
  \end{proof}
\begin{remark}
 The main body of \cite{LeungLi-functorialproperties} is devoted to a complicated proof of  the Key Lemma therein   with respect to a general  $P$. The above proposition can be obtained as an easy consequence. Nevertheless,  we  provide a detailed proof for $P=P_{\beta}$  for both the sake of completeness and the purpose of  exposition of the Key Lemma.
\end{remark}

\subsubsection{Proof of Theorem  \ref{thmfiltration}}

    For the first  statement, it  suffices to show
                    $\sigma^w\star  \sigma^u\alpha^Iq_\lambda \in F_{\mathbf{a}+\mathbf{b}}$,
                        for any $\sigma^w, \sigma^u\alpha^Iq_\lambda\in QH^*_T(G/B)$ with $\mathbf{a}=gr_\beta(w,\mathbf{0},\mathbf{0})$ and $\mathbf{b}=gr_\beta(u, I, \lambda)$. Clearly, it holds when   $\ell(w)=0$, for which $\sigma^w=\sigma^{\scriptsize\mbox{id}}$ is the unit in $ QH^*_T(G/B)$.
           We use induction on $\ell(w)$.
            If $\ell(w)=1$, then $w=s_i$ and it is done by Proposition \ref{propforreduction}.
        Assume $\ell(w)>1$ now.
      Take  $v\in W$ and  $i\in\{1,\cdots, n\}$, such that  $gr_\beta(w, \mathbf{0}, \mathbf{0})=gr_\beta(v, \mathbf{0},\mathbf{0})+gr_\beta(s_i,\mathbf{0},\mathbf{0})$ and that the coefficient of $\sigma^w$ in the cup product $\sigma^v\cup \sigma^{s_i}$ is nonzero. \big(Precisely, if
          $\mbox{sgn}_\beta(w)=1$, then we take $s_i=s_\beta$ and $v=ws_\beta$.  If
          $\mbox{sgn}_\beta(w)=0$, then we write $w=s_jv\in W^{P_\beta}$ with $\ell(v)=\ell(w)-1$, and simply take $\alpha_i\in \Delta\setminus\{\beta\}$  such that $\langle \chi_i, \gamma^\vee\rangle >0$, which exists by noting   $\gamma:=v^{-1}(\alpha_j)\neq \beta$.\big)
     %%(otherwise we would have $w=s_jv=vs_\beta\notin W^{P_\beta}$).
     By the induction hypothesis, we have $$\sigma^{s_i}\star(\sigma^{v}\star\sigma^u\alpha^Iq_\lambda) \in \sigma^{s_i}\star F_{gr_\beta(v, \mathbf{0},\mathbf{0})+\mathbf{b}}
         \subset  F_{\mathbf{a}+\mathbf{b}}.$$ On the other hand, we have
         $$(\sigma^{s_i}\star \sigma^{v})\star \sigma^u\alpha^Iq_\lambda=(\langle \chi_i, \gamma^\vee\rangle \sigma^w+\sum c^{s_i, v}_{w', I', \lambda'}\sigma^{w'}\alpha^{I'}q_{\lambda'} )\star \sigma^u\alpha^Iq_\lambda$$ with
         $\langle \chi_i, \gamma^\vee\rangle>0$ and all the coefficients $c^{s_i, v}_{w', I', \lambda'}\geq0$. There will be no cancelation, when we expand the product, due to the positivity (see Proposition \ref{propPositivity}).
             Hence, we conclude $\sigma^w    \star \sigma^u\alpha^Iq_\lambda\in F_{\mathbf{a}+\mathbf{b}}$.

             It follows directly from Definition \ref{defgradingequiv} that there is a unique term $\sigma^u\alpha^Iq_\lambda$ of grading $(m, 0)$ for every nonnegative integer $m$. It is given by $q_{\beta^\vee}^{m\over 2}$ if $m$ is even, or $\sigma^{s_\beta}q_{\beta^\vee}^{m-1\over 2}$ otherwise.
             By Proposition \ref{propQChevalley}, $$\sigma^{s_\beta}\star\sigma^{s_\beta}=q_{\beta^\vee}+\beta \sigma^{s_\beta}+\sum \langle \chi_\beta, s_\beta(\alpha^\vee)\rangle \sigma^{s_\alpha s_\beta},$$
             the summation   over those simple roots  $\alpha$   adjacent to $\beta$ in the Dynkin diagram of $\Delta$.
              Thus we have $\overline{\sigma^{s_\beta}}\star\overline{\sigma^{s_\beta}}=\overline{q_{\beta^\vee}}$   in $Gr^{\mathcal{F}}(QH^*_T(G/B))$. That is $\Psi_{\scriptsize\mbox{ver}}^\beta$ is an isomorphism of (graded) algebras (with respect to the given gradings on both sides).

As remarked earlier,   Lemma \ref{lemmalifting} implies that   $\Psi_{\scriptsize\mbox{hor}}^\beta$ is an    injective morphism    of $S$-modules. For any $\sigma^{w'}\alpha^Iq_{\lambda}$ in $QH^*_T(G/B)$ of grading $(0, *)$, we have
     $\mbox{sgn}_\beta(w')+\langle \beta,\lambda\rangle=0$. This implies that $w:=w'\in W^{P_\beta}$ if $\langle \beta, \lambda\rangle =0$, or
        $w:=w's_\beta\in W^{P_\beta}$ and $\langle \beta, \lambda\rangle =-1$ otherwise.  Denote $\lambda_P:=\lambda+\mathbb{Z}\beta^\vee \in Q^\vee/Q^\vee_{P_\beta}$. Then   $\sigma^{w'}\alpha^Iq_{\lambda}=\psi_\beta(\sigma^w\alpha^Iq_{\lambda_P})$. Thus $\Psi_{\scriptsize\mbox{hor}}^\beta$ is a bijection. Let $u, v\in W^{P_\beta}$. Consequently, we have  $\Psi_{\scriptsize\mbox{hor}}^\beta(\sigma^u \star_{P_\beta} \sigma^v)=
                    \Psi_{\scriptsize\mbox{hor}}^\beta(\sigma^u)\star  \Psi_{\scriptsize\mbox{hor}}^\beta(\sigma^v)$, by Proposition \ref{propcomparison}.
   For any $\mu_P\in Q^\vee/Q^\vee_{P_\beta}$,   $\Psi_{\scriptsize\mbox{hor}}^\beta(q_{\mu_P})$ equals $q_{\mu_B}$ if $\langle\beta, \mu_B\rangle=0$, or   $\sigma^{s_\beta}q_{\mu_B}$ otherwise. Thus
           $ \Psi_{\scriptsize\mbox{hor}}^\beta(q_{\lambda_P}\star_{P_\beta}q_{\mu_P})= \Psi_{\scriptsize\mbox{hor}}^\beta(q_{\lambda_P})\star  \Psi_{\scriptsize\mbox{hor}}^\beta(q_{\mu_P})$. By Proposition \ref{propforreduction},
        $\Psi_{\scriptsize\mbox{hor}}^\beta(q_{\mu_P} \star_{P_\beta} \sigma^v)
                  = \Psi_{\scriptsize\mbox{hor}}^\beta(q_{\mu_P})\star \Psi_{\scriptsize\mbox{hor}}^\beta(\sigma^v)$.
       Hence,   $\Psi_{\scriptsize\mbox{hor}}^\beta$ is an   isomorphism of $S$-algebras.

\subsubsection{Proof of Theorem \ref{thmforQtoC}}
     The first half  is a direct consequence of Theorem \ref{thmfiltration}.

  Now we assume the hypothesis in the second half of the statement, and consider the $\mathbb{Z}^2$-filtration $\mathcal{F}$ on $QH^*_T(G/B)$ with respect to  $\beta:=\alpha_k$.
    Write
        $$\sigma^u\star\sigma^v=\sum_{w, \lambda}N_{u, v}^{w, \lambda}\sigma^wq_\lambda=\sum_{w, I, \lambda}c_{w, I, \lambda}\sigma^w\alpha^Iq_\lambda$$ where $N_{u, v}^{w, \lambda}=\sum_Ic_{w, I, \lambda} \alpha^I$ is nonzero only if $|I|=\ell(u)+\ell(v)-\ell(w)-\langle 2\rho, \lambda\rangle\geq 0$.
    Thus in  $Gr^\mathcal{F}(QH^*_T(G/B))$, we have
     $$\overline{\sigma^u}\star\overline{\sigma^v}=\sum c_{w, I, \lambda}\overline{\sigma^w\alpha^Iq_\lambda}$$ with $\mbox{sgn}_\beta(w)+\langle \beta,\lambda\rangle=\mbox{sgn}_\beta(u)+\mbox{sgn}_\beta(v)=2$.
  Since $\mbox{sgn}_\beta(W)=\{0, 1\}$,
  $\mbox{sgn}_\beta(u)=\mbox{sgn}_\beta(v)=1$. Thus $u':=us_\beta$, $v':=vs_\beta$ are both in $W^{P_\beta}$.
By Proposition \ref{propforreduction}, %in $Gr^\mathcal{F}(QH^*_T(G/B))$ we have
  $$\overline{\sigma^{u'}\star \sigma^{s_\beta}}=\overline{ \sigma^u}\in Gr^\mathcal{F}_{gr_\beta(u,\mathbf{0},\mathbf{0})}
       \quad\mbox{and}\quad\overline{\sigma^{v'}\star \sigma^{s_\beta}}=\overline{ \sigma^v}\in Gr^\mathcal{F}_{gr_\beta(v, \mathbf{0},\mathbf{0})}.$$
  Since  the graded algebra    $Gr^\mathcal{F}(QH^*_T(G/B))$ is     associative and commutative,
      $$\overline{\sigma^u}\star\overline{\sigma^v}=\big(\overline{\sigma^{u'}}\star \overline{\sigma^{v'}}\big)\star \big(
    \overline{\sigma^{s_\beta}}\star \overline{ \sigma^{s_\beta}}\big)=
     \Psi_{\scriptsize\mbox{hor}}^\beta(\sigma^{u'}\star_{P_\beta}\sigma^{v'})\star\, \overline{ q_{\beta^\vee}}, $$
 following from   Theorem \ref{thmfiltration}. In  $QH^*_T(G/P_\beta)$, we write
        $$\sigma^{u'}\star_{P_\beta}\sigma^{v'}=\sum_{w', \lambda_{P_\beta}}N_{u', v'}^{w', \lambda_{P_\beta}}\sigma^{w'}q_{\lambda_{P_\beta}}=\sum_{w',I',  \lambda_\beta}\tilde c_{w', I', \lambda_{P_\beta} }\sigma^{w'}\alpha^{I'}q_{\lambda_{P_\beta}}.$$
 Then we have
         $$\overline{\sigma^u}\star\overline{\sigma^v}=\sum \tilde c_{w',I', \lambda_{P_\beta}} \overline{\psi_{\beta}(\sigma^{w'}q_{\lambda_{P_\beta}})\alpha^{I'} q_{\beta^\vee}}$$
  Hence,    the second half of the statement follows, by comparing   coefficients of both expressions of   $\overline{\sigma^u}\star\overline{\sigma^v}$.

  Indeed, we note    $\langle \beta, \lambda-\beta^\vee\rangle=-\mbox{sgn}_\beta(w)\in\{0, -1\}$. It follows that  the Peterson-Woodward lifting of  $\lambda_{P_\beta}:=\lambda+Q^\vee_{P_\beta}$ is given by $\lambda_B=\lambda-\beta^\vee$.
 Set $w':=w$ if $\mbox{sgn}_\beta(w)=0$, or $ws_\beta$ if $\mbox{sgn}_\beta(w)=1$.
 Then       $\psi_{\beta}(\sigma^{w'}q_{\lambda_{P_\beta}})q_{\beta^\vee}=\sigma^wq_{\lambda}$, and consequently
 $c_{w, I, \lambda}=\tilde c_{w', I, \lambda_{P_\beta}}$ for all $I$. Hence, by Proposition \ref{propcomparison},
     $$N_{u, v}^{w, \lambda}=\sum_Ic_{w, I, \lambda}\alpha^I=\sum_I\tilde c_{w', I, \lambda_{P_\beta}}\alpha^I=N_{u', v'}^{w', \lambda_{P_\beta}}=N_{u',v'}^{w, \lambda-\beta^\vee}.$$
    To show the remaining identities, we   %% Note for   $\hat w\in W$, we have $\overline{\sigma^{\hat w}}\star \, \overline{\sigma^{s_\beta}}=
      %%        \overline{\sigma^{\hat ws_\beta}}$, { if } $\mbox{sgn}_\beta(\hat w)=0$, or $\overline{\sigma^{\hat ws_\beta}q_{\beta^\vee}}$ otherwise.
consider  the expansion
  \begin{align*}
      \overline{ \sigma^u}\star\overline{ \sigma^v}
                                 =(\overline{\sigma^{u}} \star\overline{\sigma^{v'}})\star \overline{\sigma^{s_\beta}}
                 & =\overline{\sum N_{u, v'}^{\hat w, \hat \lambda}q_{\hat  \lambda}\sigma^{\hat w}}\star \overline{\sigma^{s_\beta}}\\
                & =\sum \hat c_{\hat w, I, \hat \lambda}\overline{\sigma^{\hat w s_\beta}\alpha^Iq_{\hat  \lambda}}+
                   \sum \hat c_{\hat w, I, \hat \lambda}\overline{\sigma^{\hat w s_\beta}\alpha^Iq_{\hat  \lambda +\beta^\vee}},
  \end{align*}
  where   $\mbox{sgn}_\beta(\hat w)+\langle \beta, \hat \lambda\rangle =1$
  and $\mbox{sgn}_\beta(\hat w)=0$ (resp. $1$) hold in the former (resp. latter) summation.
       Hence, if $\mbox{sgn}_\beta(w)=0$, then for every $I$
   we have $ c_{w, I,  \lambda}\overline{\sigma^w\alpha^Iq_{\lambda} }=\hat c_{\hat w, I, \hat \lambda} \overline{\sigma^{\hat w s_\beta}\alpha^Iq_{\hat  \lambda +\beta^\vee}} $
   for a unique term in the latter  summation, i.e., for $(\hat w, I, \hat \lambda)=(ws_\beta, I, \lambda-\beta^\vee)$. Thus
    $N_{u, v}^{w, \lambda}=\sum_Ic_{w,  I, \lambda}\alpha^I=\sum_I\hat c_{ws_\beta, I, \lambda-\beta^\vee}\alpha^I=N_{u, vs_\beta}^{ws_\beta, \lambda-\beta^\vee}$.
   If $\mbox{sgn}_\beta(w)=1$, then for every $I$
   we have $ c_{w, I,  \lambda}\overline{\sigma^w\alpha^Iq_{\lambda}}= \hat c_{\hat w, I, \hat \lambda}  \overline{\sigma^{\hat w s_\beta}\alpha^Iq_{\hat  \lambda  }}$
   for a unique term in the  former summation, i.e., for $(\hat w, I, \hat \lambda)=(ws_\beta,I, \lambda)$. In this case,
     $N_{u, v}^{w, \lambda}=\sum_Ic_{w,  I, \lambda}\alpha^I=\sum_I\hat c_{ws_\beta, I, \lambda}\alpha^I=N_{u, vs_\beta}^{ws_\beta, \lambda}$.

 \section{Application: an equivariant quantum Pieri rule for  $F\ell_{n_1, \cdots, n_k; n+1}$ % $SL(n+1, \mathbb{C})/P$
 }

Throughout the rest of the present paper, we let $G=PSL(n+1, \mathbb{C})$, which is the quotient group of $\tilde G=SL(n+1, \mathbb{C})$ by its center $Z(\tilde G)$. We make the Dynkin diagram of $\Delta$  in the standard way:
    \begin{tabular}{l} \raisebox{-0.4ex}[0pt]{$  \circline\!\;\!\!\circ\cdots\, \circline\!\!\!\;\circ $}\\
                 \raisebox{1.1ex}[0pt]{${\hspace{-0.2cm}\scriptstyle{\alpha_{1}}\hspace{0.3cm}\alpha_{2}\hspace{0.8cm}\alpha_{n} } $}
  \end{tabular}\! . The standard Borel subgroup  $B$ of $G$ is the quotient of the subgroup   of upper triangular matrices in $\tilde G$ by $Z(\tilde G)$. Each proper parabolic subgroup $P\supset B$ is in one-to-one correspondence with a proper subset $$\Delta_P=\Delta\setminus \{\alpha_{n_1}, \cdots, \alpha_{n_k}\},\quad\mbox{where}\quad n_0:=0< n_1<n_2<\cdots <n_k< n+1=:n_{k+1}.$$ Then
     $F\ell_{n_1, \cdots, n_k; n+1}:=PSL(n+1, \mathbb{C})/P$
     parameterizes partial flags in $\mathbb{C}^{n+1}$:
     $$F\ell_{n_1, \cdots, n_k; n+1}=\{V_{n_1}\leqslant \cdots  \leqslant V_{n_{k}}\leqslant \mathbb{C}^{n+1}~|~\dim_\mathbb{C} V_{n_i}=n_i, i=1, \cdots, k\}.$$
For each $i$,
%%an  exact sequence of tautological bundles over the complex Grassmannian $Gr(n_i, n+1)=F\ell_{n_i; n+1}$:
%% $$0\rightarrow \mathcal{S}_{(n_i)}\rightarrow \mathbb{C}^{n+1}\rightarrow \mathcal{Q}_{(n_i)}\rightarrow 0.$$
%%With
we denote by
$\pi_i: F\ell_{n_1, \cdots, n_k; n+1}\rightarrow Gr(n_i, n+1)$ the natural projection.
The equivariant quantum cohomology ring   $QH^*_T(F\ell_{n_1, \cdots, n_k; n+1})$  is generated (see e.g. \cite{AnCh,LaSh-QDSP}) by  \textit{special} Schubert classes  $\sigma^{c[n_i, p]}$ {where}
$$   c[n_i, p]:=s_{n_i-p+1}\cdots s_{n_i-1}s_{n_i}.$$
In this section,  we will show  an equivariant quantum Pieri rule for $F\ell_{n_1, \cdots, n_k; n+1}$, giving   the equivariant quantum multiplication by    $\sigma^{c[n_i, p]}$.
%%As a consequence, one can obtain  the equivariant quantum multiplication by  the by the $\mathbb{Z}^2$-symmetry of the Dynkin diagram.

 \subsection{Equivariant quantum Pieri rule}\label{subsecequivFlagAtype}
In order to state  the formula, we need  to   introduce  some   notions, mainly following    \cite{Robinson,Buch-PartialFlag,CFon}.

  The Weyl group $W$ for $PSL(n+1, \mathbb{C})$ is   isomorphic to the permutation group $S_{n+1}$, by mapping each simple reflection $s_i$ to the transposition $(i(i+1))$.
In particular, each reflection $s_\gamma$ from a positive root   $\gamma=\alpha_i+\alpha_{i+1}+\cdots +\alpha_j$
       is   sent to the transposition $(i(j+1))$, where $1\leq i\leq j\leq n$. Furthermore,
 Schubert classes $\sigma^w$ in $QH^*_T(F\ell_{n_1, \cdots, n_k; n+1})$ are indexed by $w\in W^P$ with $$W^P=\{w\in S_{n+1}~|~ w(n_{i-1}+1)<w(n_{i-1}+2)<\cdots <w(n_i), i=1, \cdots, k+1\}.$$
For   $Gr(m, n+1)=F\ell_{m; n+1}$, we have a bijection $\varphi_{m}:W^{P}\overset{\simeq}{\longrightarrow}\mathcal{P}_{m, n+1}$ to the partitions $$\mathcal{P}_{m, n+1}:=\{(a_1, \cdots, a_m)\in \mathbb{Z}^m~|~ n+1-m\geq a_1\geq a_2\geq \cdots\geq a_m\geq 0\},$$
   \begin{align}\label{eqn-WPpartition}
      w\mapsto \varphi_m(w)=(w(m)-m, \cdots, w(2)-2, w(1)-1).
   \end{align}
 We simply call  such  $w$   an \textit{$m$-th Grassmannian permutation}, whenever  $n+1$ is well understood.
 Set $\mathcal{P}_{0, n+1}:=\{(0)\}$.
Review that the length of $u\in W=S_{n+1}$ is  given by
   $$\ell(u)=|\{(i, j)~|~ 1\leq i<j\leq n+1 \mbox{ and } u(i)>u(j)\}|.$$

\begin{defn}
   Let $\zeta%=t_{i_1 r}t_{i_2r}\cdots t_{i_p r}
      =(ri_p\cdots i_2i_1)$ be a $(p+1)$-cycle in $W$. For any $u\in W$, we say that $u\zeta$ is special $j$-superior to $u$ of degree $p$ if all the following hold:
   $$(1)\,\, i_1, \cdots, i_p\leq j<r, \quad (2)\,\, u(r)>u(i_1)>\cdots > u(i_p), \quad (3)\,\, \ell(u\zeta)=\ell(u)+p.$$
More generally, if $\zeta_1, \cdots, \zeta_d$ are pairwise disjoint cycles such that each $u\zeta_s$ is special $j$-superior to $u$ of degree $p_s$ and $\sum_{s=1}^{d}p_s=p=\ell(u\zeta_1\cdots \zeta_d)-\ell(u)$, then we say that $u\zeta_1\cdots \zeta_d$ is \textbf{special $j$-superior to $u$ of degree $p$}, and denote
   $$S_{j, p}(u):=\{w\in W~|~ w \mbox{ is special } j\mbox{-superior to } u \mbox{ of degree } p\}.$$

Furthermore for  $w=u\zeta_1\cdots \zeta_d\in S_{j, p}(u)$ above, we sort  the values  $$\{u(1),\cdots, u(j)\}\setminus \{u(i)~|~  i\mbox{ occurs in some } \xi_s\}$$ to get a decreasing sequence  $[\mu_1+{j-p},  \cdots, \mu_{j-p-1}+2,  \mu_{j-p}+1]$, and then obtain
  an associated partition
    $$\mu_{w, u, j}:=(\mu_1, \mu_2, \cdots, \mu_{j-p})\in  \mathcal{P}_{j-p, n+1}.$$
  We denote  the set of such associated partitions as
    $$\mathcal{P}S_{j, p}(u):=\{\mu_{w, u, j}~|~ w \in S_{j, p}(u)\}\subset \mathcal{P}_{j-p, n+1}.$$
 \iffalse
  Furthermore for  $w=u\zeta_1\cdots \zeta_d\in S_{j, p}(u)$ above,    we define  (note $j$ is   well understood)
  {\upshape $$\mbox{Ind}(w, u):=\{i~|~ i\mbox{ \textit{occurs in some}  }   \zeta_s, \,\, i\leq j\},$$}
 %the associated indices of $w\zeta_1\cdots \zeta_d>w$.
 and   define a   subset of   partitions
   $$\mathcal{P}S_{j, p}(u):=\{\mu_{w, u, j}~|~ w \in S_{j, p}(u)\}\subset \mathcal{P}_{j-p}$$
%here we say   $u=[u(1)u(2)\cdots u(n+1)]$  in one-line notation for permutations, then
with  elements of it given by (in the following  we mean the zero partition $(0)$ if $p=j$)
 $$\mu_{w, u, j}:=(d_{j-p}-(j-p), \cdots, d_2-2, d_1-1)\in  \mathcal{P}_{j-p},$$
 where $[d_{j-p}, \cdots, d_2,  d_{1}]$ is the     decreasing sequence   obtained by  sorting the values {\upshape $\{u(i)~|~ i\leq j, \,\,\,i\notin \mbox{Ind}(w, u)\}$}. (Note {\upshape $|\mbox{Ind}(w, u)|=p=\ell(w)-\ell(u)$}.)
\fi
\end{defn}
\begin{remark}
 {\upshape
 If $p=j$, then $\mu_{w, u, j}=(0)$ is   the zero partition.
 }
\end{remark}

\begin{example} {\upshape For $PSL(7, \mathbb{C})/P=Fl_{2, 4; 7}$, we take the same $u=[3715246]\in W^P$ in  one-line notation as in Example 2 of \cite{Buch-PartialFlag}.
   Since   $w:=[4725136]=u(35)(16)$ is in $S_{4, 2}(u)$,  $1, 3, 5, 6$ are the indices occurring in $u^{-1}w$. Sorting
      $\{u(1), \cdots, u(4)\}~\setminus~$ $\{u(1), u(3), u(5), u(6)\}$, we obtain a decreasing sequence $[7, 5]$. Hence, the associated partition is given by
     $\mu_{w, u, 4}=(7-2, 5-1)=(5, 4)\in \mathcal{P}_{2, 7}$, which   corresponds to the $2$nd Grassmannian permutation $[5712346]$ for $Gr(2, 7)$.
     }
 \end{example}

So far there have been no manifestly positive formulas for general structure coefficients $N_{u, v}^{w, 0}$ of an   equivariant product $\sigma^{u}\circ \sigma^v$ of $H^*_T(G/P)$ except for the case of complex Grasssmannians and two-step flag varieties. However, there does be one  for the special case   $N_{w, v}^{w, 0}$ (i.e., when $u=w$)   in terms of a linear combination of products of positive roots
   \cite{KoKu,Billey} (for $G$ of general Lie type). Geometrically, $G/P$ has finitely many $T$-fixed points parameterized by the minimal length representatives in $W^P$. We let $\iota_w: \mbox{pt}\to G/P$ denote the natural inclusion of the $T$-fixed point labeled by $w$ into $G/P$. Then we have $N_{w, v}^{w, 0}=\iota_w^*(\sigma^v)\in H^*_T(\mbox{pt})$,  localizing  the equivariant Schubert class $\sigma^v$ at such $T$-fixed point. We will reduce all the relevant coefficients in our equivariant quantum Pieri rule to such kind of coefficients for $G/P=Gr(m, n+1)$ with $v=1^p$ being a special partition, for which there are  much more manifestly positive formulas. With the partitions in $\mathcal{P}_{m, n+1}$, we  give a precise description of $\xi^{m, p}(\mathbf{a}):=N_{\mathbf{a}, 1^p}^{\mathbf{a}, 0}$  following \cite{KoKu,Billey}.
\begin{prop-defn}\label{defnKosPolyforCoeff}
  Let $0\leq m\leq n$ and $\mathbf{a}=(a_1, \cdots, a_m)\in \mathcal{P}_{m, n+1}$. Denote by  $\mathbf{a}^T=(a_1^T, \cdots, a_{n+1-m}^T)\in \mathcal{P}_{n+1-m, n+1}$   the transpose of the partition $\mathbf{a}$. Then $s_{i_1}s_{i_2}\cdots s_{i_{|\mathbf{a}|}}$ gives a reduced expression of $\varphi_m^{-1}(\mathbf{a})\in W$,  where $|\mathbf{a}|=\sum_{s=1}^ma_s$ and
  \begin{align*}
     [i_1, \cdots, i_{|\mathbf{a}|}]:&=[n-a_{n+1-m}^{T}+1, n-a_{n+1-m}^T+2, \cdots, n;\\
                                     &\quad\,\,\, (n-1)-a_{n-m}^T+1, (n-1)-a_{n-m}^T+2,\cdots, n-1; \\
                                     &\quad\,\,\, \cdots;\\
                                     &\quad\,\,\, m-a_1^T+1, m-a_1^T+2,\cdots, m].
  \end{align*}
  Consequently,      $\gamma_b:=s_{i_1}\cdots s_{i_{b-1}}(\alpha_{i_b})$ is a positive root for any    $1\leq b\leq |\mathbf{a}|$.
  Furthermore, we have $\xi^{m, 0}(\mathbf{a})=1$; for  $1\leq p\leq m$,   the structure coefficient $\xi^{m, p}(\mathbf{a})=N_{\mathbf{a}, 1^p}^{\mathbf{a}, 0}$   is
  a homogeneous polynomial of degree $p$ in $\mathbb{Z}_{\geq 0}[\alpha_1, \cdots, \alpha_n]$ given by
         $$\xi^{m, p}(\mathbf{a}) =\sum \gamma_{j_1}\cdots \gamma_{j_p},$$
where  the sum is   over all subsequences $1\leq j_1<\cdots <j_p\leq  |\mathbf{a}|$ that satisfy   $[i_{j_1}, \cdots, i_{j_p}]=[m-p+1, m-p+2, \cdots, m]$.

\end{prop-defn}

\begin{remark}\label{rmkXivanish}
 {\upshape  %Each $\gamma_b$ above is a positive root, and hence $\xi^{m, p}(\mathbf{a})\in \mathbb{Z}_{\geq 0}[\alpha_1, \cdots, \alpha_n]$.
 If $p>a_1^T$, then $\xi^{m, p}(\mathbf{a})=0$,   for which there does not exist     $[j_1, \cdots, j_p]$ satisfying the constraint.
 }
\end{remark}
\begin{example}
{\upshape Let $n=6$. For  $\mathbf{a}:=(5,4)\in \mathcal{P}_{2, 7}$,   we have $\mathbf{a}^T=(2,2,2,2,1)\in\mathcal{P}_{5, 7}$  and
   $[i_1, \cdots, i_{|\mathbf{a}|}]=[6,4,5,3,4,2,3,1,2]$. Hence,
    \begin{align*}
       \xi^{2, 1}(\mathbf{a})&=s_6s_4s_5s_3s_4(\alpha_2) + s_6s_4s_5s_3s_4s_2s_3s_1(\alpha_2)=\alpha_1+2\alpha_2+2\alpha_3+2\alpha_4+\alpha_5+\alpha_6,\\
       \xi^{2, 2}(\mathbf{a})&=s_6s_4s_5s_3s_4s_2s_3(\alpha_1)\cdot s_6s_4s_5s_3s_4s_2s_3s_1(\alpha_2)=(\alpha_1+\cdots+\alpha_4)(\alpha_1+\cdots+\alpha_6).
         %(\sum\nolimits_{i=1}^4\alpha_i)(\sum\nolimits_{j=1}^6\alpha_j)
    \end{align*}
}
\end{example}

\begin{defn}\label{defnPieriSeq}
  Let $u\in W^P$, $1\leq i\leq k$ and $1\leq p\leq n_i$.  We denote by
     {\upshape $\mbox{Pie}_{i, p}(u)$} %the set of  \textbf{Pieri sequences} with respect  to  $(i, p, u)$, namely the set of elements
   the set of elements  $\mathbf{d}=(d_1, \cdots, d_k)$ with $[d_0, d_1, \cdots, d_k, d_{k+1}]$   being of the form
     $$[\underbrace{0,\cdots, 0},\underbrace{1,\cdots,  1},  \underbrace{2,\cdots, 2},\cdots, \underbrace{m, \cdots  m},\cdots,\underbrace{2,\cdots,2},\underbrace{1,\cdots, 1}, \underbrace{0,\cdots, 0}]$$
     that satisfy both $d_i=m\leq p$ and the next property %\quad (we denote $d_0=d_{k+1}:=0$ whenever needed) %\max\{u(n_{h_j}+1), u(n_{h_j}+2),\cdots, u(n_{l_j+1})\}
       {\upshape $$%\mbox{a)}\,\,\, d_i=m\leq p\,,\,  \mbox{ b)}\,\,
             (*):\quad u(n_{h_j})>\max\{u(r)~|~ n_{h_j}+1\leq r  \leq   n_{l_j+1}\}\mbox{ \textit{for all} } 1\leq j\leq m.$$}
   \noindent Here $1\leq h_1<\cdots<h_m\leq l_m<\cdots<l_1\leq k$ denote all the jumps, namely  $d_{h_j}=d_{l_j}=j$ and $d_{h_{j}-1}=d_{l_j+1}=j-1$ for all $1\leq j\leq m$.
  %(We simply mean $\mathbf{d}=\mathbf{0}\in$  {\upshape $\mbox{Pie}(n_i, p, u)$} when $m=0$.)
  (Note $d_0=d_{k+1}=0$.)

  Given the above $\mathbf{d}$, we denote by
      $\tau_\mathbf{d}\in S_{n+1}$   the unique permutation   defined by
         $$\tau_\mathbf{d}(n_{l_j+1}-j+1)=n_{h_j}, \,\,\,j=1, \cdots, m,$$
      together with the property that the restriction of $\tau_{\mathbf{d}}$ on the remaining elements $
      \{1, \cdots, n+1\}\setminus     \{n_{l_j+1}-j+1~|~ 1\leq j\leq m\}$   preserves the usual order.
     Similarly, we denote by    $\phi_\mathbf{d}\in S_{n+1}$   the unique permutation  given by
          $$\phi_\mathbf{d}(n_{l_j}-j+1)=n_{h_{j}-1}+1,\,\,\,\, j=1, \cdots,  m,$$
    together with the property that
    $\phi_{\mathbf{d}}\big|_{\{1, \cdots, n+1\}\setminus   \{n_{l_j}-j+1~|~ 1\leq j\leq m\}}$   preserves the usual order. In addition, we
      denote %(where we  mean $W^P$ if $\mathbf{d}=\mathbf{0}$)
       {\upshape $$\mbox{Per}(\mathbf{d}):=\{w\in W^P~|~ w(n_{h_{j}-1}+1)<\min\{w(r)~|~ n_{h_{j}-1}+2\leq r\leq n_{l_j}+1\},\,\, \forall\, j\}.$$}
 \end{defn}

The permutations
  $\tau_\mathbf{d}$ and $\phi_\mathbf{d}$ \footnote{$\tau_{\mathbf{d}}$ coincides with the permutation $\gamma_d$ in \cite{Buch-PartialFlag}. With notations in \cite{CFon}, $\tau_\mathbf{d}=\gamma_{\mathbf{hl}}$ and $\phi_{\mathbf{d}}=\delta_{\mathbf{hl}}^{-1}$.} can be expressed in terms of products of simple reflections (\cite{Buch-PartialFlag,CFon}): %That is,
  $\tau_\mathbf{d}=\tau^{(m)}\cdots \tau^{(1)}$ and $\phi_\mathbf{d}=\phi^{(m)}\cdots \phi^{(1)}$, where
         $$\tau^{(j)}=s_{n_{h_j}}\cdots s_{n_{l_j+1}-2}s_{n_{l_j+1}-1}, \qquad \phi^{(j)}=s_{n_{h_j-1}+1}\cdots s_{n_{l_j}-2}s_{n_{l_j}-1},$$
for each    $1\leq j\leq m$.
  Moreover, the above expressions are reduced, implying
   $$\ell(\tau_\mathbf{d})=\sum_{j=1}^m (n_{l_j+1}-n_{h_j})\quad\mbox{and}\quad \ell(\phi_\mathbf{d})=-m+\sum_{j=1}^m (n_{l_j}-n_{h_j-1}).$$

 \begin{remark}
 {\upshape  For $\mathbf{d}\in \mbox{Pie}_{i, p}(u)$, $d_i=m<{k\over 2}+1$. When $m=0$, we have   $\mathbf{d}=\mathbf{0}\in \mbox{Pie}_{i, p}(u)$, $\tau_{\mathbf{0}}=\phi_{\mathbf{0}}=\mbox{id}\in S_{n+1}$ and $\mbox{Per}( \mathbf{0})=W^P$.
 }
 \end{remark}
 \begin{example}\label{exampleproduct}
    {\upshape Let $PSL(7, \mathbb{C})/P=Fl_{2, 4; 7}$. Take $u=[3715246]\in W^P$, $i=2$, $p=3$. Then $\mathbf{d}\in \mbox{Pie}_{2, 3}(u)$ only if
      $\mathbf{d}=(0, 0), (0,1)$ or $(1, 1)$. For $(0, 1)$, the jumps are given by $h=l=2$. Since $\max\{u(5), u(6), u(7)\}=6>5=u(n_2)$, $(*)$ is not satisfied. For $(1, 1)$, we have $d_2=m=\max\{d_1, d_2\}=1<3=p$ and $1=h<l=2$. Clearly, $u(n_1)=u(2)=7>\max\{u(3), \cdots, u(7)\}$. Thus,  $\mbox{Pie}_{2, 3}(u)=\{(0, 0), (1, 1)\}.$
    }
 \end{example}

  For  $F\ell_{n_1, \cdots, n_k; n+1}$, there are $k$ quantum variables $q_{\alpha_{n_i}^\vee+Q^\vee_P}$, $1\leq i\leq k$,  which we simply denote as $\bar q_i$ respectively.
   %%%%Denote $\bar {{q}}^{\mathbf{d}}:= \bar q_1^{d_1}\cdots \bar q_k^{d_k}$. We notice   $\bar {{q}}^{\mathbf{0}}=1\in \mathbb{Z}$.
   %We have the next \textit{equivariant quantum Pieri rule}.

\begin{thm}[Equivariant quantum Pieri rule for $F\ell_{n_1, \cdots, n_k; n+1}$]\label{thmequiQPR} For any $1\leq i\leq k, 1\leq p\leq n_i$ and any $u\in W^P$, we have
{\upshape $$\sigma^{c[n_i, p]}\star \sigma^u=\sum_{(d_1, \cdots, d_k)\in \scriptsize\mbox{Pie}_{i, p}(u)}  \sum_{j=0}^{p-d_i}  \sum_{w} \xi^{n_i-d_i-j,  p-d_i-j}(\mu_{w\cdot \phi_\mathbf{d}, u\cdot \tau_\mathbf{d}, n_i-d_i})\sigma^w  \bar q_1^{d_1}\cdots \bar q_k^{d_k},$$
}
where  the last sum is  over all {\upshape $w\in \mbox{Per}(\mathbf{d})$} satisfying
     $w\cdot \phi_{\mathbf{d}}\in S_{n_i-d_i, j}(u\cdot \tau_\mathbf{d})$.
 \end{thm}

Let us say a few words on the constraints in the   theorem.    Given $\mathbf{d}=(d_1, \cdots, d_k)$ of the form in   Definition \ref{defnPieriSeq} with $d_i=\max\{d_1, \cdots, d_k\}\leq p$, we  have (see   \cite{Buch-PartialFlag,CFon}).
\begin{align*}
  \mathbf{d}\in  \mbox{Pie}_{i, p}(u) &\Longleftrightarrow \ell(u\cdot\tau_\mathbf{d})=\ell(u)-\ell(\tau_\mathbf{d});\\
    w\in \mbox{Per}(\mathbf{d})&\Longleftrightarrow \ell(w\cdot \phi_\mathbf{d})=\ell(w)+\ell(\phi_\mathbf{d}).
\end{align*}

\begin{remark}
   The above formula  is different from the one given in   \cite{LaSh-affinePieri} by Lam and Shimozono,  who worked on the side of   equivariant homology of affine Grassmannians. In \cite{LaSh-affinePieri},   special Schubert classes are of the form $\sigma^{s_ps_{p-1}\cdots s_2s_1s_\theta}$ where $\theta$ denotes the highest root, and they generate $QH^*_T(F\ell_{1, \cdots, n; n+1})$ as well. These special classes, in general,   are not pullback from $H^*(F\ell_{n_1, \cdots, n_k; n+1})$, and therefore  do not induce  equivariant quantum Pieri rules for
      $F\ell_{n_1, \cdots, n_k; n+1}$ immediately.
\end{remark}
\iffalse
 Furthermore, a coefficient $\xi^{n_i-d_i-j,  p-d_i-j}(\mu_{w\cdot \phi_\mathbf{d}, u\cdot \tau_\mathbf{d}, n_i-d_i})$ vanishes if and only if
      $p-d_i-j$ is larger than the first entry $\mu_1^T$ of the transpose  $(\mu_1^T, \cdots, \mu_{n+1-n_i+d_i+j}^T)$ of  the partition $\mu_{w\cdot \phi_\mathbf{d}, u\cdot \tau_\mathbf{d}, n_i-d_i}\in \mathcal{P}_{n_i-d_i-j, n+1}$.
\fi
\noindent\textbf{Example \ref{exampleproduct} (Continued). } %Note   %$\mbox{Pie}_{2, 3}(u)=\{(0, 0), (1, 1)\}$,
%%In order to calculate $c_3^{4}\star \sigma^u$,  $\tau_{(1, 1)}=(234567)$,
%%  $\phi_{(1, 1)}=(1234)$, and
%% $c_3^{4}\star \sigma^u= \sum_{j=0}^{3}  \sum_{w} \xi^{4-j,  3-j}(\mu_{w,  u, 4})\sigma^w +$  $\sum_{j=0}^{2}  \sum_{w} \xi^{3-j,  2-j}(\mu_{w\cdot {\phi_{(1, %%1)}},  u\cdot\tau_{(1, 1)}, 3})\sigma^w\bar q_1\bar q_2,$
Note  %$\mbox{Pie}_{2, 3}(u)=\{(0, 0), (1, 1)\}$,
          $\tau_{(1, 1)}=(234567)$,
   $\phi_{(1, 1)}=(1234)$, $u\cdot\tau_{(0,0)}=u=[3715246]$ and $u\cdot \tau_{(1,1)}=[3152467]$.
Write $w\cdot \phi_\mathbf{d}=(u\cdot\tau_{\mathbf{d}}) \cdot \zeta$  for     $w\in \mbox{Per}(\mathbf{d})\bigcap \big(S_{4-d_2, j}(u\cdot \tau_\mathbf{d})\big)\cdot (\phi_{\mathbf{d}})^{-1}$.  Denote $\mathbf{a}:=\mu_{w\cdot \phi_\mathbf{d}, u\cdot \tau_\mathbf{d}, 4-d_2}$.  Precisely, we have
  \begin{center}
  \begin{tabular}{|c|c|c|c|c|c|c|c|}
    \hline
    $\mathbf{d}$& $j$&$w\cdot \phi_\mathbf{d}$& $\zeta$ & $\mathbf{a}$ &  $[i_1,\cdots,i_{|\mathbf{a}|}]$& $w$  \\ \hline \hline
     &  $0$ & $[3715246]$ & $\mbox{id}$  & $(3,2,1,0)$ &   $[6,4,5,2,3,4]$&    \\ \cline{2-6}
     &    & $[4715236]$ & $(16)$  & $(4,3,0)$ &    $[6,4,5,3,4,2,3]$&   \\ \cline{3-6}
    &  $1$ & $[3725146]$ & $(35)$  & $(4,3,2)$ &  $[6,4,5,2,3,4,1,2,3]$&   \\ \cline{3-6}
    &   & $[3716245]$ & $(47)$  & $(4,1,0)$ &   $[6,5,4,2,3]$& \raisebox{1.5ex}[0pt]{Coincide}   \\ \cline{2-6}
    \raisebox{1.5ex}[0pt]{$(0,0)$} &   & $[4725136]$ & $(16)(35)$  & $(5,4)$ &  $[6,4,5,3,4,2,3,1,2]$&  \raisebox{1.5ex}[0pt]{with}   \\ \cline{3-6}
     &  $2$ & $[4716235]$ & $(16)(47)$  & $(5,0)$ &  $[6,5,4,3,2]$& \raisebox{1.5ex}[0pt]{$w\cdot \phi_\mathbf{d}$}   \\ \cline{3-6}
      &    & $[3726145]$ & $(35)(47)$  & $(5,2)$ &  $[6,5,4,2,3,1,2]$&    \\ \cline{2-6}
    &  $3$ & $[4726135]$ & $(16)(35)(47)$  & $(6)$ &  $[6,5,4,3,2,1]$&    \\ \hline
  %%  $(1,1)$ &  $0$ & $[3152467]$ & $\mbox{id}$  & $(2,1,0)$ &   $[4,2,3]$& $[2315467]$  \\ \hline
  %%  $(1,1)$ &  $1$ & $[4152367]$ & $(15)$  & $(3,0)$ &    $[4,3,2]$& $[2415367]$  \\ \hline
     &  $1$ & $[3251467]$ & $(24)$  & $(3,2)$ &   $[4,2,3,1,2]$& $[1325467]$  \\ \cline{2-7}
  %% $(1,1)$ &  $1$ & $[3162457]$ & $(36)$  & $(1,0)$ &  $[2]$& $[2316457]$  \\ \hline
   $(1,1)$ &   & $[4251367]$ & $(15)(24)$  & $(4)$ &    $[4,3,2,1]$& $[1425367]$  \\ \cline{3-7}
   %% $(1,1)$ &  $2$ & $[4162357]$ & $(15)(36)$  & $(0)$ &   $\emptyset$& $[2416357]$  \\ \hline
   &  \raisebox{1.5ex}[0pt]{$2$} & $[3261457]$ & $(24)(36)$  & $(2)$ &  $[2,1]$& $[1326457]$  \\ \hline
  \end{tabular}
 \end{center}
  By Definition \ref{defnKosPolyforCoeff}, we can write down  $\xi^{4-d_i-j, 3-d_i-j}(\mathbf{a})$  immediately. By abuse of notation, we  simply denote each Schubert class $\sigma^v$ as $v$. In conclusion, we have
   \begin{align*}&\quad{c[4, 3]}\star [3715246]\\
     &= \alpha_2(\alpha_2+\alpha_3+\alpha_4)(\alpha_2+\cdots+\alpha_6)[3715246]\\
      &\quad+\alpha_2(\alpha_2+\cdots+\alpha_6)[3716245]+(\alpha_2+\alpha_3+\alpha_4)(\alpha_2+\cdots+\alpha_6)[4715236]\\
       &\quad+\big(\alpha_2(\alpha_2+\alpha_3+\alpha_4)+\alpha_2(\alpha_1+\cdots+\alpha_6)+(\alpha_1+\cdots+\alpha_4)(\alpha_1+\cdots+\alpha_6)\big)[3725146]\\
          &\quad  +(\alpha_1+2\alpha_2+\alpha_3+\alpha_4+\alpha_5+\alpha_6)[3726145] +(\alpha_2+\cdots+\alpha_6)[4716235] \\
          &\quad +(\alpha_1+2\alpha_2+2\alpha_3+2\alpha_4+\alpha_5+\alpha_6)[4725136]\\
                        &\quad+[4726135]+\bar q_1\bar q_2[1326457]+\bar q_1\bar q_2[1425367]\\
          &\quad+(\alpha_1+2\alpha_2+\alpha_3+\alpha_4)\bar q_1\bar q_2[1325467]\\
   \end{align*}
% \bigskip

\iffalse
A  presentation of the equivariant quantum cohomology ring    $QH^*_T(F\ell_{n_1, \cdots, n_k; n+1})$     was first given by Kim \cite{Kim-EquiQH}.
The Giambelli-type formula, which express Schubert classes in terms of polynomial of the generators in the ring presentation, has been  given by Lam-Shimozono \cite{LaSh-QDSP} and Anderson-Chen \cite{AnCh}  recently.
There is likely another approach to both results,  by using the equivariant quantum Pieri  rule above. In the present paper, we will deal with the special case of  complex Grassmannians in next section, by using  such approach. This provides an alternative proof of the main results in \cite{Mihalcea-EQGiambelli}.
 %%(or more precisely, by   finding the equivariant analogue of equation (1.3), (1.4) of \cite{CFon} and studying the corresponding equivariant quantum products parallel to section 3 of \cite{CFon}).
\fi

 \subsection{Proof of the equivariant quantum Pieri rule for $F\ell_{n_1, \cdots, n_k; n+1}$}

%The rest
This subsection is devoted to a proof of Theorem \ref{thmequiQPR}.
We will show  it  by   reducing all the relevant structure coefficients   to certain structure coefficients  of degree zero using  Theorem \ref{thmforQtoC}, so that we can  apply  Robinson's equivariant Pieri rule \cite{Robinson}.

\subsubsection{Robinson's equivariant Pieri rule}

  For $u\in W$ and   $w\in S_{r, j}(u)$ where $0\leq j\leq r\leq n$, we denote by   $\{i_1<i_2<\cdots<i_j\}$ the set of indices $i_s$ such that $i_s<r$ and $i_s$ occurs in a cycle decomposition of $u^{-1}w$. Here we mean the empty set if $j=0$. In \cite{Robinson}, Robinson introduced an associated element
   $v_{[w, u, r]}=[v(1)\cdots  v(n+1)]$ that is obtained  from $u$ by moving the entries $u(i_1), \cdots, u(i_j)$ to positions $r-j+1, r-j+2, \cdots, r$, respectively, and preserving the relative positions of all other entries. That is, $v_{[w, u, r]}$ is of the form  $[*\cdots * u(i_1)u(i_2)\cdots u(i_j)u(r+1)\cdots u(n+1)]$.   The next   equivariant Pieri rule for $F\ell_{n+1}:=F\ell_{1,2, \cdots, n; n+1}$ is due to Robinson.

\begin{prop}[Theorem A of \cite{Robinson}]\label{propEquivPieri}
     Let $u\in W$ and $1\leq p\leq r\leq n$. %%Denote $c[r,  p]:=s_{r-p+1}s_{r-p+2}\cdots s_r$.
      We have
        $$\sigma^{c[r, p]}\circ \sigma^u=\sum_{j=0}^p\sum_{w\in S_{r, j}(u)} N_{c[r-j,  p-j],v_{[w, u, r]}}^{v_{[w, u, r]}, 0}\sigma^w\quad\mbox{in } H^*_T(F\ell_{n+1}).$$
\end{prop}

\iffalse
\begin{remark}
   The inequality $c[r-j, p-j]\leq v_{[w, u, r]}$, with respect to the Bruhat order,  may not hold. Hence, some coefficients
   $\xi^{c[r-j,  p-j]}(v_{[w, u, r]})$ above  may vanish.
\end{remark}
\fi

In the following, we further reduce the structure coefficients in the above equivariant Pieri rule to a more special type   for complex Grassmannians.

\begin{cor}\label{coreEquiPieri}
  Let $1\leq i\leq k$, $1\leq p\leq n_i$ and $u\in W^P$. In $H^*_T(F\ell_{n_1, \cdots, n_k; n+1})$, %we have
     $$\sigma^{c[n_i, p]}\circ \sigma^u=\sum_{j=0}^p\sum_{w\in S_{n_i, j}(u)} \xi^{n_i-j,  p-j}(\mu_{w, u, n_i})\sigma^w.$$
Furthermore, a coefficient $\xi^{n_i-j,  p-j}(\mu_{w, u, n_i})$ vanishes if and only if
      $p-j$ is larger than the first entry $\mu_1^T$ of the transposed partition $\mu_{w, u, n_i}^T=(\mu_1^T, \cdots) \in \mathcal{P}_{n+1-n_i+j, n+1}$.
\end{cor}

\begin{proof}
We let      $\tilde v\in S_{n+1}$ be the $(n_i-j)$-th Grassmannian permutation (which has   at most a descent at the $(n_i-j)$-th position)  determined by the property that
$[\tilde v(1)\cdots \tilde v(n_i-j)]$ is an  increasing sequence obtained from $u$ by sorting
     the values $$\{u(1), \cdots, u(n_i)\}\setminus \{u(d)~|~ d\leq n_i, d \mbox{ occurs in a cycle decomposition of } u^{-1}w\}.$$
    Then  by definition, $\mu_{w, u, n_i}=\varphi_{n_i-j}(\tilde v)$ is a partition in $\mathcal{P}_{n_i-j, n+1}$.
 Moreover,
  $x:=\tilde v^{-1} v_{[w,u,n_i]}$ is in the Weyl subgroup      generated by  $\{s_\alpha~|~ \alpha\neq {\alpha_{n_i-j}}, \alpha\in \Delta\}$, and     $\ell(v_{[w,u,n_i]})=\ell(\tilde v)+\ell(x)$. We notice that $\mbox{sgn}_\alpha(c[n_i-j,  p-j])$  for any $\alpha\neq \alpha_{n_i-j}$. It follows immediately from Corollary \ref{correduction} (1)
that $$ N_{c[n_i-j,  p-j], v_{[w, u, n_i]}}^{v_{[w, u, n_i]}, 0}=N_{c[n_i-j,  p-j], \tilde vx}^{\tilde v x, 0}=N_{c[n_i-j,  p-j], \tilde v}^{\tilde v, 0}=\xi^{n_i-j, p-j}(\mu_{w, u, n_i}).$$
Write the transpose of   $\mu_{w, u, n_i}$ as $(\mu_1^T, \mu_2^T, \cdots, \mu_{n+1-n_i+j}^T)$. It follows directly from Proposition \ref{defnKosPolyforCoeff}
%(see also Remark \ref{rmkXivanish})
that $\xi^{n_i-j, p-j}(\mu_{w, u, n_i})\neq 0$ if and only if $p-j\leq \mu_1^T$.
\end{proof}

\iffalse
\begin{lemma}
   Let $u\in W$ and  $1\leq p\leq r\leq n$. Then    $q_1^{d_1}\cdots q_n^{d_n}$  occurs in  $c^r_p\star \sigma^u\in QH^*_T(F\ell_{n+1})$ only if  all the following hold:  {\upshape $\mbox{i) } d_r\leq p$; $\mbox{ii) }d_0:=0\leq d_1\leq \cdots \leq d_r\geq d_{r+1}\geq \cdots \geq d_n\geq 0=:d_{n+1}$; $\mbox{iii) } |d_j-d_{j+1}|\leq 1$ for all $0\leq j\leq n$.
    }
\end{lemma}
\fi

\subsubsection{Proof of Theorem \ref{thmequiQPR} for $F\ell_{n+1}$}
In this subsection, we will prove the theorem for the case   $k=n$: $$F\ell_{n_1, \cdots, n_k; n+1}=F\ell_{1,2, \cdots, n; n+1}=F\ell_{n+1}.$$
This will be done by a combination of a number of claims. Our readers can first focus on the  statements themselves  without referring to the technical arguments, in order to get an outline of the proof of our Pieri rule.

For $F\ell_{n+1}$, there are $n$      quantum variables   $q_1, \cdots, q_n$, and $n_i=i$ for $i=1, \cdots, n$.
  The statement to prove concerns the equivariant quantum multiplication by Schubert classes $\sigma^{c[n_i, p]}$  with
 $$c[n_i, p]=s_{n_i-p+1}\cdots s_{n_i-1}s_{n_i}.$$
We notice that $\mbox{sgn}_{r}(c[n_i, p])$ is equal to $1$ if $r=n_i$, or $0$ otherwise. On the other hand,
for any nonzero effective coroot $\lambda\in Q^\vee$, there always exists a simple root $\alpha$ such that  $\langle \alpha, \lambda\rangle>0$.
In many cases, we can find $\alpha\neq \alpha_{n_i}$ with  $\langle \alpha, \lambda\rangle>1$, implying  that $q_\lambda$ never occurs in the corresponding multiplication  by Theorem \ref{thmforQtoC} (1). Therefore, $q_\lambda$ occurs in $\sigma^{c[n_i, p]}\star \sigma^u$ only if $\lambda$ is of particular type. Precisely,

 %% We investigate all the nonzero coefficients $N_{v, u}^{w, \lambda}$ in  $\sigma^v\star \sigma^u$ as follows.

\vspace{0.1cm}

\noindent\textbf{Claim A: }{\itshape
Assume  $N_{c[n_i, p], u}^{w, \lambda}\neq 0$, where $\lambda=d_1\alpha_1^\vee+\cdots+d_n\alpha_n^\vee$. Then  we have
   $$(1)\,\,  d_i\leq p;\quad (2)\,\, 0\leq d_1\leq\cdots \leq d_i;\quad  (3)\,\,  d_i\geq\cdots \geq d_n\geq 0.$$ }
\begin{proof}
    Clearly,  $\sigma^{c[n_i, p]}$ occurs in $(\sigma^{s_{n_i}})^p$. Since $N_{c[n_i, p], u}^{w,\lambda}\neq 0$,  $q_\lambda$ occurs in the product $(\sigma^{s_{n_i}})^p\star \sigma^u$
    by   the positivity property.  By    Proposition \ref{propQChevalley},   there is $0\leq m \leq p$ such that $\lambda=\sum_{j=1}^m(\alpha_{k_j}^\vee+\alpha_{k_j+1}^\vee+\cdots+\alpha_{k_j'}^\vee)$ where  $k_j\leq n_i\leq k_j'$ for each $j$.
     Thus   $d_i=m=\max\{d_1, \cdots, d_n\}\leq p$. This proves (1).

       If (2) did not hold, then    $\{{j}~|~j<i, d_j>  d_{j+1}\}$ is non-empty, so that we can take the minimum $b$ of it.
     Then $1\leq b<i-1$ and $d_{b-1}\leq d_b$.

      If  $d_{b-1}<d_{b}$,  then we have the following inequalities by noting   $\mbox{sgn}_b(c[n_i, p])=0$:
     $$\mbox{sgn}_{b}(w)+\langle \alpha_b, \lambda\rangle=\mbox{sgn}_b(w)+(2d_b-d_{b-1}-d_{b+1})\geq 2>\mbox{sgn}_b(c[n_i, p])+\mbox{sgn}_b(u).$$   This  would imply $N_{c[n_i, p], u}^{w, \lambda}=0$ by Theorem \ref{thmforQtoC} (1), which makes a   contradiction.

    If $d_{b-1}=d_b$, then for $a:=\min\{j~|~ d_j=d_b\}\leq b-1$, we have   $d_{a-1}<d_a$. If $d_a-d_{a-1}\geq 2$, we conclude $N_{c[n_i, p], u}^{w, 0}=0$ again by using $\mbox{sgn}_a$. If $d_a-d_{a-1}=1$, then by using Corollary \ref{correduction} (2) repeatedly,  we have $N_{c[n_i, p], us_as_{a+1}\cdots s_{b-1}}^{ws_as_{a+1}\cdots s_{b-1}, \lambda'}=N_{c[n_i, p], u}^{w, \lambda}\neq 0$, where $\lambda'=\lambda-\alpha_{a}^\vee-\cdots-\alpha_{b-1}^\vee$. Note $\langle \alpha_b, \lambda'\rangle=\langle \alpha_b, (d_b-1)\alpha_{b-1}^\vee+d_b\alpha^\vee_b+d_{b+1}\alpha^\vee_{b+1}\rangle \geq 2$. It  would follow that   $N_{c[n_i, p], us_as_{a+1}\cdots s_{b-1}}^{ws_as_{a+1}\cdots s_{b-1}, \lambda'}=0$, which makes     a contradiction again.

Hence, $0\leq d_1\leq\cdots \leq d_i$.
  Similarly,  we can show   $d_i\geq\cdots \geq d_n\geq 0$.
\end{proof}

%\vspace{0.15cm}

Thanks to the above claim, we have $m:=\max\{d_1, \cdots, d_n\}=d_i$ if $N_{c[n_i, p], u}^{w, \lambda}\neq 0$. Moreover,
we can denote by $$1\leq h_{r_1}<\cdots< h_{m-1}< h_m\leq l_m< l_{m-1}< \cdots< l_{r_2}\leq n$$
all the jumps among $d_r$'s,  namely  $d_{h_j-1}<d_{h_{j}}$ for $r_1\leq j\leq m$,  and $d_{l_j}>d_{l_j+1}$ for $m\geq j\geq r_2$.
 As we will show in Claim C, there are exactly $m$ increasing (decreasing) jumps (i.e.,  $r_1=r_2=1$), together with lots of constraints on $u$ and $w$.
To make this conclusion,  we denote
  $$\varpi_j:=s_{h_j}\cdots s_{i-2}s_{i-1}\cdot s_{l_j}\cdots s_{i+2}s_{i+1},$$
 $$\begin{array}{rlcrl}
     \qquad  \hat\tau^{(j)}:=&\!\!\!\! s_{h_j}s_{h_j+1}\cdots s_{l_j},\quad\mbox{and}  &&\hat\phi^{(j)}:=&\!\!\!\!s_{h_j}s_{h_j+1}\cdots s_{l_j-1}\\
\end{array}
$$
for any $\max\{r_1, r_2\}\leq j\leq m$. Define $\lambda^{(j-1)}$ inductively by
  $$ \lambda^{(m)}:=\lambda=d_1\alpha_1^\vee+\cdots+d_n\alpha_n^\vee \quad \mbox{and}\quad \lambda^{(j-1)}:=\lambda^{(j)}-\sum_{r=h_j}^{l_j}{\alpha_r}^\vee.$$
We will prove the conclusion by induction on $\lambda^{(j)}$.   Here is the first step of the induction, which we prove by applying  Corollary \ref{correduction} repeatedly.
 \vspace{0.1cm}

\noindent\textbf{Claim B: }{\itshape $N_{c[n_i, p], u}^{w, \lambda^{(m)}}\neq 0$ if and only if all the following hold:
\begin{enumerate}
                \item $d_{h_m-1}=d_{l_m+1}=m-1$;
                \item $\ell(u\hat\tau^{(m)})=\ell(u)-\ell(\hat\tau^{(m)})$ and $\ell(w\hat\phi^{(m)})=\ell(w)+\ell(\hat\phi^{(m)})$;
                  \item $N_{c[n_i-1, p-1], u\hat\tau^{(m)}}^{w\hat\phi^{(m)}, \lambda^{(m-1)}}\neq 0$;
                  \item  $N_{c[n_i-1, p-1], u\hat\tau^{(m)}}^{w\hat\phi^{(m)}, \lambda^{(m-1)}}=N_{c[n_i, p], u}^{w, \lambda^{(m)}}; \qquad (5)\,\, \ell(u\varpi_ms_i)=\ell(u)-\ell(\varpi_ms_i)$.
 \end{enumerate}}
\begin{proof}
   We first assume  $N_{c[n_i, p], u}^{w, \lambda^{(m)}}\neq 0$ and  discuss all the possibilities as follows.

    i) Assume  $h_m=l_m=i$ (i.e., both an increasing jump and a decreasing jump happen at the $(i=n_i)$-th position).  If (1) did not hold, then   $\langle \alpha_i, \lambda\rangle=(m-d_{i-1})+(m-d_{i+1})>2$. This would   imply  $N_{c[n_i, p], u}^{w, \lambda}=0$, making a contradiction. Hence, (1) holds, and $\langle \alpha_i, \lambda\rangle=2$. In this case, $\hat\tau^{(m)}=\hat\phi^{(m)}=s_i$ and $\varpi_m=\mbox{id}$. Hence, all (2), (3), (4), (5)  follow immediately from Theorem \ref{thmforQtoC}.

      ii)  Assume $h_m<i$ and $l_m=i$. Since $N_{c[n_i, p], u}^{w, \lambda}\neq 0$, it follows %from  Theorem \ref{thmforQtoC} (1)
                     that
      $$1\geq \mbox{sgn}_{h_m}(u)=\mbox{sgn}_{h_m}(c[n_i, p])+\mbox{sgn}_{h_m}(u)\geq \mbox{sgn}_{h_m}(w)+\langle \alpha_{h_m}, \lambda\rangle\geq m-d_{h_m-1}\geq 1.$$
     Hence,   all the inequalities are in fact equalities. Thus we have $d_{h_m-1}=m-1$, $\ell(us_{h_m})=\ell(u)-1$, $\ell(ws_{h_m})=\ell(w)+1$, and consequently $N_{c[n_i, p], us_{h_m}}^{ws_{h_m}, \lambda-\alpha_{h_m}^\vee}=N_{c[n_i, p], u}^{w, \lambda}\neq 0$ by   Corollary \ref{correduction} (2).
      For    $h_m<a<n_i=l_m$, we note $\mbox{sgn}_a(c[n_i, p])=0$ and $\langle \alpha_a, \lambda-\alpha_{h_m}^\vee-\cdots -\alpha_{a-2}^\vee-\alpha_{a-1}^\vee\rangle=1$.
  Using Corollary \ref{correduction} (2) repeatedly,  we conclude   $\ell(us_{h_m}s_{h_m+1}\cdots s_{i-1})=\ell(u)-
     (i-h_m)$,   $\ell(ws_{h_m}s_{h_m+1}\cdots s_{i-1})=\ell(w)+(i-h_m)$, and $N_{c[n_i, p], us_{h_m}s_{h_m+1}\cdots s_{i-1}}^{ws_{h_m}s_{h_m+1}\cdots s_{i-1},  \lambda-\alpha_{h_m}^\vee-\cdots -\alpha_{i-2}^\vee-\alpha_{i-1}^\vee}=N_{c[n_i, p], u}^{w, \lambda}\neq 0$.  Since the reduced structure coefficient is nonzero,   $$2\geq \langle \alpha_i, \lambda-\alpha_{h_m}^\vee-\cdots -\alpha_{i-2}^\vee-\alpha_{i-1}^\vee\rangle=1+m-d_{i+1}\geq 2.$$
     Hence,   $d_{l_m+1}=d_{i+1}=m-1$, and consequently and
      $$N_{c[n_i-1, p-1], u\hat\tau^{(m)}}^{w\hat\phi^{(m)}, \lambda^{(m-1)}}=  N_{c[n_i, p], us_{h_m}s_{h_m+1}\cdots s_{i-1}}^{ws_{h_m}s_{h_m+1}\cdots s_{i-1},  \lambda-\alpha_{h_m}^\vee-\cdots -\alpha_{i-2}^\vee-\alpha_{i-1}^\vee} \neq 0$$
      by Theorem \ref{thmforQtoC} (2). That is, the statements (1), (3), (4) hold. It is easy to see that (2) and (5) hold as well.

     iii)  Assume $h_m=i$ and $l_m>i$. The claim holds by similar arguments to ii).

     iv)  Assume $h_m<i$ and $l_m>i$. Again by similar arguments to ii), we conclude $d_{h_m-1}=d_{l_m+1}=m-1$, % (i.e., (1) holds),
     $\ell(u\varpi_ms_i)=\ell(u)-\ell(\varpi_ms_i)$,   $\ell(w\varpi_m)=\ell(w)+\ell(\varpi_m)$,
       $$\mbox{and}\hspace{2cm}0\neq N_{c[n_i, p], u}^{w, \lambda}=N_{c[n_i, p], u\varpi_m}^{w\varpi_m, \lambda^{(m-1)}+\alpha_i^\vee}=N_{c[n_i-1, p-1], u\varpi_ms_i}^{w\varpi_m, \lambda^{(m-1)}}.\hspace{5cm}{}$$
       For every  $i+1\leq a\leq l_m$, we have
          $\mbox{sgn}_a(c[n_i-1, p-1])=0$, $\langle\alpha_a, \lambda^{(m-1)}\rangle=0$ and $$\ell(ws_{h_m}\cdots s_{i-2}s_{i-1} \cdot s_{l_m}\cdots s_{a+1}s_{a})=
           \ell(ws_{h_m}\cdots s_{i-2}s_{i-1}\cdot s_{l_m}\cdots s_{a+2}s_{a+1})+1.$$
       By Corollary  \ref{correduction} (1), we conclude
      $$N_{c[n_i-1, p-1], u\varpi_ms_i}^{w\varpi_m, \lambda^{(m-1)}}=N_{c[n_i-1, p-1], u\varpi_ms_is_{i+1}\cdots s_{l_m}}^{ws_{h_m}\cdots s_{i-2}s_{i-1}, \lambda^{(m-1)}}\!\!\mbox{ and }\,\, \ell(u\varpi_ms_is_{i+1}\cdots s_{l_m})=\ell(u\varpi_ms_i)-l_m+i.$$
  Note $\varpi_ms_is_{i+1}\cdots s_{l_m}= \hat\tau^{(m)}s_{l_m-1}\cdots s_{i+1}s_i$. It follows that
  $$\ell(u\hat\tau^{(m)}s_{l_m-1}\cdots s_{i+1}s_i)=\ell(u)-(l_m-h_m+1)-l_m+i=\ell(u)-\ell(\hat\tau^{(m)})-l_m+i .$$
  As a consequence,  we have $\ell( u \hat\tau^{(m)})=\ell(u)-\ell(\hat\tau^{(m)})$, and   $\ell(u\hat\tau^{(m)}s_{l_m-1}\cdots s_{b+1} s_b)=\ell(u\hat\tau^{(m)}s_{l_m-1}\cdots s_{b})+1$ for all $i\leq b\leq l_m-1$. Hence, by Corollary \ref{correduction} (1), we have
  $$N_{c[n_i-1, p-1], u\hat\tau^{(m)}s_{l_m-1}\cdots s_{i+1}s_i}^{ws_{h_m}\cdots s_{i-2}s_{i-1}, \lambda^{(m-1)}}=N_{c[n_i-1, p-1], u\hat\tau^{(m)}}^{w\hat\phi^{(m)}, \lambda^{(m-1)}} $$
      $$\mbox{and}\hspace{1cm}\ell(w\hat\phi^{(m)})=\ell(ws_{h_m}\cdots s_{i-2}s_{i-1})+l_m-i= \ell(w)+\ell(\hat\phi^{(m)}). \hspace{2cm}{}$$
      In a summary, all (1)--(5) hold.

    The other direction is obvious. (In fact, $(4)$ is a consequence of  the hypotheses (1), (2) and (3).)
\end{proof}

 By using claims A and   B, the next claim follows immediately by induction on $\lambda^{(j)}$.

 \vspace{0.2cm}

  \noindent\textbf{Claim C: }{\itshape    $N_{c[n_i, p], u}^{w, \lambda}\neq 0$ only if   all the  following hold:
            \begin{enumerate}
               \item[(a)] $r_1=r_2=1$, namely there are exactly  $2m$ jumps among   $[0, d_1, \cdots, d_n, 0]$.
               \item[(b)] $\ell(u\cdot\hat\tau^{(m)}\cdots\hat\tau^{(1)})=\ell(u)-\ell(\hat\tau^{(m)}\cdots\hat\tau^{(1)})$ and \\ $\ell(w\cdot\hat\phi^{(m)}\cdots\hat\phi^{(1)})=\ell(w)+\ell(\hat\phi^{(m)}\cdots\hat\phi^{(1)})$.
               \item[(c)]   $\ell(u\varpi_ms_i\varpi_{m-1}s_{i-1}\cdots \varpi_1s_{i-m+1})=\ell(u)-\sum_{j=1}^m\ell(\varpi_js_{i-m+j})$.
            \end{enumerate}
        Whenever  both  {\upshape (a)} and  {\upshape (b)}  hold,  we   have
             $$N_{c[n_i-m, p-m], u\cdot \hat\tau^{(m)}\cdots\hat\tau^{(1)}}^{w\cdot \hat\phi^{(m)}\cdots\hat\phi^{(1)}, 0} =N_{c[n_i, p], u}^{w, \lambda} %\qquad(\mbox{possibly equal to zero})
             .$$
             }
 \vspace{0.2cm}
Since $n_i=i$ in the case of $F\ell_{n+1}$, we have     $\hat\tau^{(j)}= \tau^{(j)}$
               and  $\hat\phi^{(j)}=\phi^{(j)}$ for all $j$. Therefore we have  $\hat\tau^{(m)}\cdots \hat\tau^{(1)}=\tau_{\mathbf{d}}$ and $\hat\phi^{(m)}\cdots \hat\phi^{(1)}=\phi_{\mathbf{d}}$.
     Hence, we finish the proof of Theorem \ref{thmequiQPR} for $F\ell_{n+1}$, by  using  Corollary \ref{coreEquiPieri} together with the fact that
    the   hypotheses (a), (b) in Claim C  are equivalent to   the hypotheses     $\mathbf{d}\in \mbox{Pie}_{i, p}(u)$ and
              $w\in \mbox{Per}(\mathbf{d})$.

 \vspace{0.25cm}
 In order to show the general case in   next subsection, we make one more claim.
 \vspace{0.25cm}

   \noindent\textbf{Claim D: }{\itshape
   Let $1\leq p\leq r\leq n$ and $u\in W$.  Suppose $q_1^{d_1}\cdots q_{n}^{d_n}$ occurs in the product $\sigma^{c[r, p]}\star  \sigma^u$ in  $QH^*_T(F\ell_{n+1})$.
        If  $j$ coincides with some jump $h_b$ or $l_b$ with  $1\leq b<m$, then we have  $\ell(us_j)<\ell(u)$.
}

\vspace{0.25cm}

\begin{proof}
By the hypothesis, there exists $w\in W$ such that $N_{c[r, p], u}^{w, \lambda}\neq 0$, where $\lambda=d_1\alpha_1^\vee+\cdots +d_n\alpha_n^\vee$.
By    Claim C (b), (c), we have
   $\ell(u\cdot\hat\tau^{(m)}\cdots\hat\tau^{(1)})=\ell(u)-\ell(\hat\tau^{(m)}\cdots\hat\tau^{(1)})$ and  $\ell(u\varpi_ms_r\varpi_{m-1}s_{r-1}\cdots \varpi_1s_{r-m+1})=\ell(u)-\sum_{j=1}^m\ell(\varpi_js_{r-m+j})$.
   Thus we have
        $\ell(us_a)<\ell(u)$, whenever a reduced expression of $\hat \tau^{(m)}\cdots \hat \tau^{(1)}$ or $\varpi_ms_r\cdots \varpi_1s_{r-m+1}$
                   starts with $s_a$. Hence, we are done, due to the following:
 % Note $h_b <h_{b+1}$. By direct calculations,
    $$(\hat\tau^{(m)}\cdots \hat\tau^{(1)})^{-1}(\alpha_{j})=\begin{cases}
      -(\alpha_{l_{b+1}-b+1}+\cdots\alpha_{l_b-b}+\alpha_{l_b-b+1})&\mbox{if }h_b=h_{b+1}-1\\
          -(\alpha_{h_{b}-b+1}+\cdots\alpha_{l_b-b}+\alpha_{l_b-b+1})&\mbox{if }h_b<h_{b+1}-1\end{cases}.$$
  Hence, $\hat\tau^{(m)}\cdots \hat\tau^{(1)}$ admits a reduced           expression  starting with $s_{h_b}$.
  Similarly, we conclude $(\varpi_ms_r\cdots \varpi_1s_{r-m+1})^{-1}(\alpha_{l_b})\not\in R^+$  by direct calculations. Consequently,   $\varpi_ms_r\cdots \varpi_1s_{r-m+1}$  admits a reduced           expression  starting with $s_{l_b}$.
\iffalse     Note the expression of $\hat\tau^{(m)}\cdots \hat\tau^{(1)}$ is reduced, and  $\hat\tau^{(m)}\cdots \hat\tau^{(1)}=\hat \tau^{(b)} \hat\tau^{(m)}_{\scriptstyle [1]}\cdots \hat\tau^{(b+1)}_{\scriptstyle [1]}\cdot
                              \hat \tau^{(b-1)}\cdots \hat \tau^{(1)}$, where $\hat\tau^{(a)}_{\scriptstyle [1]}:=s_{h_a-1}s_{{h_a}}\cdots s_{l_a-1}$. Thus                                the expression on the right hand side is also reduced, and it    starts with $s_{h_b}$. Thus if $j=h_b$, we have $\ell(us_j)<\ell(u)$.

     Denote $w_a':=s_{h_{a}}s_{h_{a}+1} \cdots s_{i-m+a-1}$ and $w_a'':=s_{l_a}s_{l_a-1}\cdots s_{i-m+a+1}$. That is, $w_a=w_a'w_a''=w_a''w_a'$ for all $1\leq a\leq m$.
     Since $l_{b}>l_m\geq m$,   $w_{b}''\neq \mbox{id}$. Note $w_{b+1}s_{i-m+b+1}w_{b}s_{i-m+b}=w_{b+1}'(w_{b+1}''s_{i-m+b+1}w''_b)w_b's_{i-m+b}=
       w_{b+1}'w_b''\cdot (s_{l_{b+1}+1}$ $s_{l_{b+1}}\cdots  s_{i-m+b+2})\cdot w'_bs_{i-m+b}=
       s_{l_b}s_{l_b-1}\cdots s_{i-m+b+2}\cdot w'_{b+1}s_{i-m+b+1}\cdot (s_{l_{b+1}+1}s_{l_{b+1}}$ $\cdots  s_{i-m+b+2})\cdot w'_bs_{i-m+b}$.
 Note the given expression of $w_ms_i\cdots w_1s_{i-m+1}$   is  reduced. Thus both sides of the above subexpressions are reduced, and it starts with $s_{l_b}s_{l_b-1}\cdots s_{i-m+b+2}$ on the right hand side.  By induction, we conclude
       $w_as_{i-m+a}\cdots w_1s_{i-m+1}$ has a reduced expression starting with $s_{l_b}s_{l_b-1}\cdots s_{i-m+a+1}$ for all $b\leq a\leq m$.
    Thus if $j=l_b$, we also have $\ell(us_j)<\ell(u)$.
 \fi
\end{proof}

\subsubsection{Proof of Theorem \ref{thmequiQPR} for $F\ell_{n_1, \cdots, n_k; n+1}$} In this subsection, we   prove the theorem for    general  $F\ell_{n_1, \cdots, n_k; n+1}$ by reducing all the relevant structure coefficients to  the case of  $F\ell_{n+1}$, thanks to the equivariant Peterson-Woodward comparison formula.
We will    use  $\,\,\bar {}\,\,$ to distinguish from the  notations for $F\ell_{n+1}$.
For instance, we denote by $\bar q_j=q_{\alpha_{n_j}^\vee+Q^\vee_P}$ the quantum variables in  $QH^*_T(F\ell_{n_1, \cdots, n_k; n+1})$.
  For  $\lambda_P=\sum_{j=1}^k \bar d_j\alpha_{n_j}^\vee+Q^\vee_P$ and $u, w\in W^P$,
  by Proposition \ref{propcomparison}  we have
   $$N_{c[n_i, p], u}^{w, \lambda_P}=N_{c[n_i, p], u}^{w\omega_P\omega_{P'}, \lambda_B}$$
   for a unique   $\lambda_B=d_1\alpha_1^\vee+\cdots+ d_n\alpha_n^\vee$.

We investigate all the nonzero $N_{c[n_i, p], u}^{w, \lambda_P}$. By Claim C (a), we can    denote all the jumps  of the sequence $[0, d_1, \cdots, d_n, 0]$ as
     $1\leq h_1<\cdots<h_m\leq l_m<\cdots <l_1\leq n$.
   Since $\lambda_P=\lambda_B+Q_P^\vee$, we have  $d_{n_j}=\bar d_j$ for all $1\leq j\leq k$.

\vspace{0.25cm}

\noindent\textbf{Claim E: }{\itshape Assume $N_{c[n_i, p], u}^{w, \lambda_P}\neq 0$. Then
    there are   $2m$ jumps in total  among the sequence $[0, \bar d_1, \cdots, \bar d_k, 0]$:
   $$1\leq \bar h_1<\cdots <\bar h_m\leq \bar l_m<\cdots \bar l_1\leq k,$$
which are  given by the jumps for $\lambda_B$. Precisely, for all $1\leq j\leq m$,  we have
 $$h_j=n_{\bar h_j}, \quad l_j=n_{\bar l_{j}}, \mbox{ and } \bar d_{\bar h_j}=d_{h_j}=d_{l_j}=\bar d_{\bar l_j}=j.$$
}

\vspace{0.25cm}

\begin{proof}
It follows from Claim A and Claim B (1) that $d_{h_m}=m\geq d_{h_m+1}\geq m-1$ and  $d_{h_m-1}=m-1$.
      Hence,  $\langle \alpha_{h_m}, \lambda_B\rangle \in \{1, 2\}$. Since    $\langle \alpha, \lambda_B\rangle \in\{0, -1\}$ for all $\alpha\in\Delta_P$, we have  $\alpha_{h_m}\not\in \Delta_P$. Thus  $h_m\in\{n_1,\cdots, n_k\}$, i.e.,  $h_m=n_{\bar h_m}$ for some $1\leq \bar h_m\leq k$.
  Similarly, we conclude $d_{l_m}=m$, $l_m\in \{n_1, \cdots, n_k\}$, and hence $l_m= {n_{\bar l_m}}$.
 It follows that   $ \bar d_i=d_{n_i}=m=\max\{d_1, \cdots, d_n\}=\max\{\bar d_1, \cdots, \bar d_k\}$, and  the first two jumps around $\bar d_i$ occur exactly on
   $\bar h_m\leq \bar l_m$.

  By Claim D,  we have  $\ell(us_a)<\ell(u)$ for all $a\in \{h_1, \cdots, h_{m-1}, l_1, \cdots, l_{m-1}\}$. This implies    $\alpha_a\not \in\Delta_P$, since $u\in W^P$.
      Thus  $\{h_1, \cdots, h_{m-1}, l_1, \cdots, l_{m-1}\}\subset \{n_1, \cdots, n_k\}$.
The statement becomes a direct  consequence of Claim C (a).
\end{proof}

   By Claim C, we have
   \begin{align*}
      N_{c[n_i, p], u}^{w\omega_P\omega_{P'}, \lambda_B}&=N_{c[n_i-m, p-m], u\hat \tau^{(m)}\cdots \hat\tau^{(1)}}^{w\omega_P\omega_{P'}\hat \phi^{(m)}\cdots \hat\phi^{(1)}, 0};\\
      \ell(u\hat \tau^{(m)}\cdots \hat\tau^{(1)})&=\ell(u)-\ell(\hat \tau^{(m)}\cdots \hat\tau^{(1)}) \\
    \mbox{and}\qquad  \ell(w\omega_P\omega_{P'}\hat \phi^{(m)}\cdots \hat\phi^{(1)})&=\ell(w\omega_P\omega_{P'})+\ell(\hat \phi^{(m)}\cdots \hat\phi^{(1)}), \qquad \mbox{with}
  \end{align*}
     $$\hat \phi^{(j)}=s_{h_j}\cdots s_{l_j-1}=s_{n_{\bar h_j}}s_{n_{\bar h_j}+1}\cdots s_{n_{\bar l_j}-1},\quad \hat \tau^{(j)}=s_{h_j} \cdots s_{l_j}=s_{n_{\bar h_j}} \cdots s_{n_{\bar l_j}}.$$
    Note $\Delta_{P'}=\{\alpha\in \Delta_P~|~\langle \alpha, \lambda_B\rangle=0\}= \Delta_P\setminus \{\alpha_{n_{h_j}-1}, \alpha_{n_{l_j}+1}~|~ j=1, \cdots, m\}$ where $\Delta_P=\Delta\setminus\{\alpha_{n_1}, \cdots, \alpha_{n_k}\}$. It follows that
    $\omega_P\omega_{P'}=u_1\cdots u_mv_m\cdots v_1$, with
     $$u_j:=s_{n_{\bar h_{j}-1}+1}s_{n_{\bar h_{j}-1}+2}\cdots s_{n_{\bar  h_j}-1} \quad\mbox{and}\quad v_j:=s_{n_{\bar l_{j}+1}-1}s_{n_{\bar l_j+1}-2}\cdots s_{n_{\bar l_j}+1}.$$  Clearly,   $u_1, \cdots, u_m$, $v_1, \cdots, v_m$ are pairwise commutative. Denote
      $$v_j^{\scriptstyle{[j-1]}}:=s_{n_{\bar l_{j}+1}-j}s_{n_{\bar l_j+1}-j-1}\cdots s_{n_{\bar l_j}-j+2},$$
      which does not contain  $s_{n_i-m}$. It follows that
 $$\omega_P\omega_{P'}\hat\phi^{(m)}\cdots \hat\phi^{(1)}=u_m\hat\phi^{(m)}\cdots u_1\hat\phi^{(1)}v_1^{\scriptstyle {[0]}}v_2^{\scriptstyle {[1]}}\cdots v_m^{\scriptstyle{[m-1]}} =\phi^{(m)}\cdots \phi^{(1)}v_1^{\scriptstyle{[0]}}v_2^{\scriptstyle{[1]}}\cdots v_m^{\scriptstyle{[m-1]}}.$$
For $\bar{\mathbf{d}}=(\bar d_1, \cdots, \bar d_k)$, we recall    $\phi_{\bar {\mathbf{d}}}=\phi^{(m)}\cdots \phi^{(1)}$. Since $w\in W^P$,  we have
   $$\ell(w\omega_P\omega_{P'}\hat\phi^{(m)}\cdots\hat \phi^{(1)})%=\ell(w)+\ell(\omega_P\omega_{P'})+\ell(\hat\phi^{(m)}\cdots\hat \phi^{(1)})
   =\ell(w)+\ell(\phi_{\bar{\mathbf{d}}})+\ell(v_1^{\scriptstyle{[0]}}v_2^{\scriptstyle{[1]}}\cdots v_m^{\scriptstyle{[m-1]}}).$$  Hence,  by Corollary \ref{correduction}(1),
   we have
   \begin{align*}
      N_{c[n_i-m, p-m], u\hat \tau^{(m)}\cdots \hat\tau^{(1)}}^{w\omega_P\omega_{P'}\hat \phi^{(m)}\cdots \hat\phi^{(1)}, 0}&=
    N_{c[n_i-m, p-m], u\hat \tau^{(m)}\cdots \hat\tau^{(1)}}^{w\phi_{\bar {\mathbf{d}}}v_1^{\scriptstyle{[0]}}v_2^{\scriptstyle{[1]}}\cdots v_m^{\scriptstyle{[m-1]}}, 0}\\
    &= N_{c[n_i-m, p-m], u\hat \tau^{(m)}\cdots \hat\tau^{(1)}\cdot (v_1^{\scriptstyle{[0]}}v_2^{\scriptstyle{[1]}}\cdots v_m^{\scriptstyle{[m-1]}})^{-1}}^{\phi_{\bar {\mathbf{d}}}, 0}\\
    &=N_{c[n_i-m, p-m], u\cdot \tau_{\bar{\mathbf{d}}}}^{w\cdot \phi_{\bar{\mathbf{d}}}, 0}
   \end{align*}

  $$\mbox{and}\quad \ell(u\cdot \tau_{\bar{\mathbf{d}}})=\ell(u)-\ell(\hat \tau^{(m)}\cdots \hat\tau^{(1)})-\ell((v_1^{\scriptstyle{[0]}}v_2^{\scriptstyle{[1]}}\cdots v_m^{\scriptstyle{[m-1]}})^{-1} )=\ell(u)-\ell(\tau_{\bar{\mathbf{d}}}).\qquad{}$$
   %%%Here we have used the observation that $\hat \tau^{(m)}\cdots \hat\tau^{(1)}\cdot (v_1^{\scriptstyle{[0]}}v_2^{\scriptstyle{[1]}}\cdots v_m^{\scriptstyle{[m-1]}})^{-1}$ $=\hat \tau^{(m)} v_m^{-1}\cdots \hat\tau^{(2)}v_2^{-1}\hat \tau^{(1)} v_1^{-1}=\tau^{(m)}\cdots \tau^{(1)}=\tau_{\bar{\mathbf{d}}}$.
Hence,   $$N_{c[n_i, p], u}^{w, \lambda_P}=N_{c[n_i, p], u}^{w\omega_P\omega_{P'}, \lambda_B}= N_{c[n_i-m, p-m], u\cdot \tau_{\bar{\mathbf{d}}}}^{w\cdot \phi_{\bar{\mathbf{d}}}, 0}\quad.$$
Then we are done by Corollary \ref{coreEquiPieri}.

\subsection{Specialization to  complex Grassmannians}

In this subsection, we will further simplify our equivariant quantum Pieri rule for the special case     $Gr(m, n+1)=F\ell_{m; n+1}$.
The bijection map
$\varphi_m: W^P\overset{\simeq}{\rightarrow}\mathcal{P}_{m, n+1}$   sends
   $$c[m, p]=s_{m-p+1}\cdots s_{m-1}s_m  \quad\mbox{to}\quad 1^p:=(1, \cdots, 1, 0, \cdots,0)\in \mathcal{P}_{m, n+1} \quad(p \mbox{ copies of }1). $$
Therefore we will also denote   the special  Schubert classes $\sigma^{c[m, p]}$  as  $\sigma^{1^p}$.
We remark  that  $\sigma^{1^p}$ is related with (but different from) the $p$-th equivariant Chern class $c_p^T(\mathcal{S}^*)$ of the dual of the tautological subbundle   (see e.g.  \cite[\S 5.1]{Mihalcea-EQGiambelli}).

The equivariant quantum multiplication by $\sigma^1$ was given by Mihalcea  \cite{Miha-EQSC}.
Here we will give a neat formula of the multiplication by all  $\sigma^{1^p}$  by simplifying   Theorem \ref{thmequiQPR}.
We remark that   the classical part of our formula, i.e., the equivariant Pieri rule, is   different from those known rules in  \cite{Laksov,GaSa}. It is obtained by simplifying Robinson's  Pieri rule in a purely combinatorial way. Nevertheless, our formulation has inspired   the second author and Ravikumar to find an equivariant Pieri rule for
 Grassmannians of all classical Lie types \cite{LiRa} in a geometric way.
 \begin{defn}\label{defpartitionforCompGrass}
   Let  $\nu=(\nu_1, \cdots, \nu_m)$ and $\eta=(\eta_1, \cdots, \eta_m)$ be partitions in $\mathcal{P}_{m, n+1}$ with $\eta_i-\nu_i\in\{0, 1\}$ for all $1\leq i\leq m$. Denote %%$r:=|\eta|-|\nu|$        and denote
    by $j_1<j_2<\cdots <j_{m-r}$ all those  $\eta_{j_i}=\nu_{j_i}$.
   We define a partition  $\eta_\nu$ in $\mathcal{P}_{m-r, n+1}$ associated to $(\eta, \nu)$ by
                  $$\eta_\nu:=(\nu_{j_1}-j_1+r+1,  \nu_{j_2}-j_2+r+2,\cdots, \nu_{j_{m-r}}-j_{m-r}+m).$$
\end{defn}

The above definition can be alternatively described by the language of Young diagrams as follows. We also provide an example illustrated by Figures 1 and 2.   

\begin{defn-eg}\label{defeg}
   Let  $\nu, \eta$ be  partitions in   $\mathcal{P}_{m, n+1}$ such that the Young diagram of $\eta$ is obtained by adding a vertical strip to the Young diagram of $\nu$. Denote by $r$   the number of boxes in the strip $\eta/\nu$. 
 We define an associated partition $\eta_\nu$ in $\mathcal{P}_{m-r, n+1}$  by a simple join-and-cut operation as follows.
  \begin{enumerate}
    \item[Step 1:] Whenever a row of the Young diagram of $\eta$   inside the $m\times (n+1-m)$ rectangle  does not contain a box in the strip $\eta/\nu$, we add $A$ boxes, where $A$ counts the remaining rows of the rectangle below the given one.  We then move them to an $(m-r)\times (n+1-m+r)$ rectangle preserving the relative positions, which could be beyond the boundary of the  rectangle on the right.
    \item[Step 2:] For each row in the $(m-r)\times (n+1-m+r)$ rectangle, we remove $B$ boxes, where $B$ counts the remaining rows of the rectangle below the given one.
  \end{enumerate}
 As a result, we obtain a partition in $\mathcal{P}_{m-r,  n+1}$, denoted as $\eta_{\nu}$. In particular if $\eta=\nu$, then $r=0$ and $\eta_\nu=\nu$.
 
 Figure 1 illustrates the case of $\nu=(6, 3, 2, 2, 0, 0)$ and $\eta=(6, 3,3, 2,1,1)$ in $\mathcal{P}_{6, 13}$, for which we have $r=3$. Then the associated   partition in $\mathcal{P}_{3, 13}$ is given by $\eta_\nu=(9, 6, 4)$ as illustrated by Figure 2.
\end{defn-eg}

\begin{figure}[h]
  \caption{Young diagrams of partitions $\eta, \nu$}
  \label{partition}
\includegraphics[scale=1]{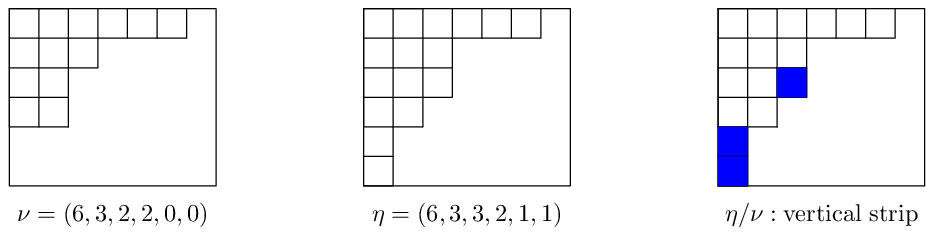}
\end{figure}

\begin{figure}[h]
  \caption{Associated parition $\eta_\nu$ by a join-and-cut operation }\label{joincut}
\includegraphics[scale=1]{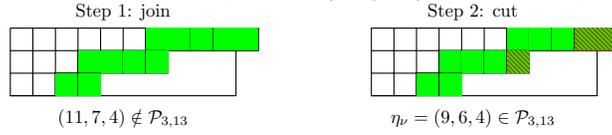}
\end{figure}

\begin{lemma}\label{lemmafromFltoGrass}
  Let $1\leq p\leq m$ and   $v, w\in W^P$. Denote  $\nu:=\varphi_m(v)$ and $\eta:=\varphi_m(w)$.
     Then $w\in S_{m, p}(v)$ if and only  if both of the following hold:
  {\upshape    $$\mbox{(i)}\,\,|\eta|=|\nu|+p;\qquad\mbox{(ii)}\,\,\eta_j-\nu_j\in \{0, 1\} \quad\mbox{ for all } 1\leq j\leq m.$$
  }
 Furthermore when this holds,
       we have   $\mu_{w, v, m}=\eta_\nu$.
\end{lemma}

\begin{proof}
We assume $w\in S_{m, p}(v)$ first.
Write $w=v\zeta_1\cdots\zeta_d$ where $\zeta_1, \cdots, \zeta_d$ are pairwise disjoint cycles.
Since $v, w\in W^P$, it follows that  $v\zeta_1\cdots\zeta_k\in W^P$ for all $1\leq k\leq d$.
Denote $p_k:=\ell(v\zeta_1\cdots\zeta_k)-\ell(v\zeta_1\cdots\zeta_{k-1})$. Then   $v\zeta_1\cdots\zeta_k\in S_{m, p_k}(v\zeta_1\cdots\zeta_{s-1})$ follows from the definition. In particular,  write $\zeta_1=(ri_{p_1}\cdots i_2i_1)$, then
  we have $$(1)\,\, i_1, \cdots, i_{p_1}\leq m< r; \qquad  (2)\,\, v(r)>v(i_1)>\cdots>v(i_{p_1}).$$

\noindent\textbf{Claim F: }  Denote $i_0:=r$. We have
 $$(a) [i_{p_1},\cdots, i_2, i_1]=[i_1-p+1,  \cdots, i_1-1, i_1];\quad (b)\,\, v(i_j)=v(i_{j+1})+1,  0\leq j\leq p_1-1.$$
Assuming the above claim first, we write   $\varphi_m(v\zeta_1)=(\eta^{(1)}_1, \cdots, \eta^{(1)}_m)=:\eta^{(1)}$.
For any  $1\leq s\leq m$ distinct from those $m+1-i_j$, we have
 $$\eta^{(1)}_s=v\zeta_1(m+1-s)-(m+1-s)=
  v(m+1-s)-(m+1-s)=\eta_s.$$
 By Claim F (b), we have
      $$\eta_{m+1-i_j}^{(1)}=v\zeta_1(i_j)-i_j=v(i_{j-1})-i_j=v(i_j)+1-i_j=\eta_{m+1-i_j}+1.$$
  Together with Claim F (a), we obtain
  $$(**):\quad \eta^{(1)}=(\nu_1, \cdots, \nu_{m-i_1},  \nu_{m-i_1+1}+1, \cdots, \nu_{m-i_1+p_1}+1, \nu_{m-i_1+p_1+2},\cdots, \nu_{m}),$$
  where   $p_1\leq i_1\leq m$. Thus if $d=1$, then  we are done.

Now we assume $d\geq 2$. Write $\zeta_2=(r'i_{p_2}'\cdots i_2'i_1')$. Since $\zeta_1, \zeta_2$ are disjoint cycles,
    $m+1-i_j\not\in\{m+1-i_1', \cdots, m+1-i'_{p_2}\}$. Write $\varphi_m(v\zeta_1\zeta_2)=\eta^{(2)}=(\eta_1^{(2)}, \cdots, \eta^{(2)}_m)$. Then
    $\eta_s^{(2)}=\eta^{(1)}_s$ whenever $s\in\{m+1-i_1, \cdots, m+1-i_{p_1}\}$.
    Using the same arguments as above, we conclude $\eta_s^{(2)}-\eta_s^{(1)}\in \{0, 1\}$ for all  $s\in\{1, \cdots, m\}\setminus\{m+1-i_1, \cdots, m+1-i_{p_1}\}$.
    Thus we  have   $\eta_s^{(2)}-\nu_s\in \{0, 1\}$ for all $1\leq s\leq m$.
   Hence, both (i) and (ii) hold   by induction on $k$.

   As a direct consequence of the above arguments,  we observe that
  $1\leq a\leq m$ occurs in some cycle $\zeta_s$ if and only if $\eta_{m+1-a}-\nu_{m+1-a}=1$.
 Hence, $[v(m+1-j_{1}), v(m+1-j_2), \cdots, , v(m+1-j_{m-p})]$ is the decreasing sequence obtained by
  sorting $\{v(1),\cdots, v(m)\}\setminus\{v(a)~|~ a \mbox{ occurs in } \zeta_s \mbox{ for some }$ $ 1\leq s\leq d\}$.
  Hence, the partition $\mu_{w, v, m}=(\mu_1, \cdots, \mu_{m-p})$ in $\mathcal{P}_{m-p, n+1}$ coincides with $\eta_\nu$, by noting $$\mu_i=v(m+1-j_i)-(m-p+1-i)=v(m+1-j_i)-(m+1-j_i)+p+i-j_i=\nu_{j_i}+p+i-j_i.$$

  On the other hand, we assume the hypotheses (i) and (ii) both hold now. If $\varphi_m(w)=\eta$ is given by  $(**)$, we define $i_j:=i_1-j+1$ for every $1\leq j\leq p_1$ and define $r$ to be the element satisfying $v(r)=v(i_1)+1$. It is easy to check $r>m$. Consequently, we have
     $w=v(ri_{p_1}\cdots i_2i_1)\in S_{m, p_1}(v)$ with $p_1=\ell(w)-\ell(v)=p$. In general, there are $d$ nests of consecutive $1$, for which we can construct  pairwise disjoint cycles $\zeta_1, \cdots, \zeta_d$ by induction, such that $w=v\zeta_1  \cdots \zeta_d\in S_{m, p}(v)$.

  It remains to show Claim F. It follows from  $v\in W^{P}$ and properties (1), (2) that $m\geq i_1>i_2>\cdots>i_{p_1}$. If (a) did not hold, then $i_j>i_{j+1}+1$ for some $1\leq j\leq p_1-1$, and  we would deduce a contradiction:
           $$v(i_j)=v\zeta_1(i_{j+1})<v\zeta_1(i_{j+1}+1)=v(i_{j+1}+1)<v(i_j).$$
  Hence, (a) holds.
  If (b) did not hold, then  $v(i_j)>v(i_{j+1})+1$ for some $0\leq j\leq p_1-1$.
  If $j>0$, then    $i_{j+1}+1=i_j$ by Claim F (a), and consequently
   %%$v(i_{j+1})+1\neq v(s)$ for any $1\leq s\leq m$. Hence,
   $v(r)>v(i_{j+1})+1=v(a)$ for some $m+1<a<r$.
       In this case,   we  deduce a contradiction: $$v(a)=v\zeta_1(a)<v\zeta_1(r)=v(i_{p_1})\leq v(i_{j+1})=v(a)-1.$$
   If $j=0$, then   $v(r)>v(i_{1})+1=v(b)$ for some $b<r$. If $b>m$, then we would have
          $v(b)=v\zeta_1(b)<v\zeta_1(r)=v(i_{p_1})\leq v(i_1)=v(b)-1$. If $b\leq m$, then $b>i_1$ and consequently we have
              $v(r)=v\zeta_1(i_1)<v\zeta_1(b)=v(b)<v(r)$. Either cases deduces a contradiction again.
 \end{proof}

Using the above lemma, we can simplify  Theorem \ref{thmequiQPR} for the special case of complex Grassmannians, and therefore obtain the following.    The proof is essentially  the same as Corollary 3.3 of \cite{CFon}.
  There is only one quantum variable
  in $QH^*_T(Gr(m, n+1))$, which we simply denote as $q$.
\begin{thm}[Equivariant quantum Pieri rule for complex Grassmannians]\label{thmQPRforComplexGrassmannian}
   Let $1\leq p\leq m$ and  $\nu=(\nu_1, \cdots, \nu_m)\in \mathcal{P}_{m, n+1}$. In $QH^*_T(Gr(m, n+1))$, we have
     {\upshape $$\sigma^{1^p}\star \sigma^\nu=  \sum_{r=0}^{p} \sum_{\eta} \xi^{m-r,  p-r}(\eta_\nu)\sigma^\eta+
      \sum_{r=0}^{p-1}  \sum_{\kappa} \xi^{m-1-r,  p-1-r}(\kappa'_{\nu'})\sigma^\kappa q,$$
}
where the second sum is over those $\eta=(\eta_1, \cdots, \eta_m)\in \mathcal{P}_{m, n+1}$ satisfying $|\eta|=|\nu|+r$ and $\eta_i-\nu_i\in\{0, 1\}$ for all $1\leq i\leq m$; the $q$-terms
 occur  only if $\nu_1=n+1-m$, and when this holds,   the last sum is over
    those $\kappa=(\kappa_1, \cdots, \kappa_{m-1},0)\in \mathcal{P}_{m, n+1}$ such that
      $\kappa':=(\kappa_1+1, \cdots, \kappa_{m-1}+1)$ and $\nu':=(\nu_2, \cdots, \nu_m)$
        satisfy $|\kappa'|=|\nu'|+r$ and $\kappa_i+1-\nu_{i+1}\in\{0, 1\}$ for all $1\leq i\leq m-1$.
\end{thm}
\begin{proof}
Denote $v=\varphi_m^{-1}(\nu)$. Using the  same notations as in Theorem \ref{thmequiQPR}, we have
     $k=1$,  and hence $\mbox{Pie}_{1, p}\subset \{(0), (1)\}$. By Theorem \ref{thmequiQPR}, we have
         $$\sigma^{1^p}\star \sigma^\nu=  \sum_{j=0}^{p} \sum_{w\in S_{m, j}}  \xi^{m-j,  p-j}(\mu_{w, v, m})\sigma^w+\epsilon
      \sum_{j=0}^{p-1}  \sum_{w} \xi^{m-1-j,  p-1-j}(\mu_{w\cdot \phi_{(1)}, v\cdot\tau_{(1)}, m-1})\sigma^w q,$$
      where    $\epsilon=1$ if $\ell(v\cdot\tau_{(1)})=\ell(v)-\ell(\tau_{(1)})$, or $0$ otherwise; the last sum is over those $w\in \mbox{Pie}((1))$ satisfying $w\cdot \phi_{(1)}\in S_{m-1, j}(v\cdot\tau_{(1)})$.

 The classical part of the formula to prove is referred to as the equivariant $\mbox{Pieri}$ rule.  It follows from the canonical injective morphism $H_T^*(Gr(m, n+1))\hookrightarrow H^*_T(F\ell_{n+1})$ that $\xi^{m-j,  p-j}(\mu_{w, v, m})\neq 0$ only if $w\in W^P$.
  Hence, the equivariant Pieri rule
 follows directly from  Lemma \ref{lemmafromFltoGrass}.

  When  $\mathbf{d}=(1)$,  we have
       $\tau_{\mathbf{d}}=s_ms_{m+1}\cdots s_n$ and $\phi_{\mathbf{d}}=s_1s_2\cdots s_{m-1}$.
   Note $v(m)=\max\{v(1), \cdots, v(m)\}$ and $v(n+1)=\max\{v(m+1), \cdots, v(n+1)\}$. As a consequence,   the following are all equivalent:
$$\mbox{i) } \ell(v\cdot\tau_{(1)})=\ell(v)-\ell(\tau_{(1)}); \mbox{ ii) }
  v\tau_{\mathbf{d}}(\alpha_n)\in R^+; \mbox{ iii) }  v(m)>v(n+1);%$vs_ms_{m+1}\cdots s_n(n+1)=v(m)>v(n+1)=vs_ms_{m+1}\cdots s_n(n)$;
   \mbox{ iv) }v(m)=n+1.$$
Hence, we have $\epsilon=1$ if and only if $\nu_1=v(m)-m=n+1-m$.  Furthermore when this holds,
  $v\cdot \tau_{(1)}$ is a Grassmannian permutation for $Gr(m-1, n+1)$, which corresponds to the partition   $\varphi_{m-1}(v\cdot \tau_{(1)})=(\nu_2, \cdots, \nu_{m-1})=:\nu'$ in $\mathcal{P}_{m-1, n+1}$ (by noting   $v\tau_{\mathbf{d}}(j)=v(j)$ for $1\leq j\leq m-1$).

    Write $\varphi_m(w)=(\kappa_1, \cdots, \kappa_m)=:\kappa$. Then the following are equivalent:
$$(1)\,\,w\in \mbox{Pie}((1)); \quad(2)\,\, \ell(w\cdot \phi_{(1)})=\ell(w)+\ell(\phi_{(1)}); \quad (3)  \kappa_m=0.$$
 It follows that     $w\cdot \phi_{(1)}$ is a Grassmannian permutation for $Gr(m-1, n+1)$, which corresponds to the partition    $\varphi_{m-1}(w\cdot \phi_{(1)})=(\kappa_1+1, \cdots, \kappa_{m-1}+1)=:\kappa'$.

     Hence, the $q$-part also becomes a direct consequence of Lemma \ref{lemmafromFltoGrass}.
       \end{proof}
\begin{remark}
  The non-equivariant quantum Pieri rule \cite{Bertram} can be obtained by using Proposition 11.10 of \cite{LaSh-GoverPaffineGr} and the Pieri-type formula of
   $H_*(\Omega SU(n+1))$ in \cite{LLMS}. It will be very interesting to generalize this approach to   the equivariant quantum cohomology of complex (or more generally, cominuscule) Grassmannians.
\end{remark}

\begin{example}
   Among the  product $\sigma^{(1,1,0)}\star \sigma^{(4, 2, 1)}$ in $H^*_T(Gr(3, 7))$,  two terms $q^0\sigma^\eta$ and $q^1 \sigma^\kappa$  can be  read off from the following figure:
   
   \includegraphics[scale=1]{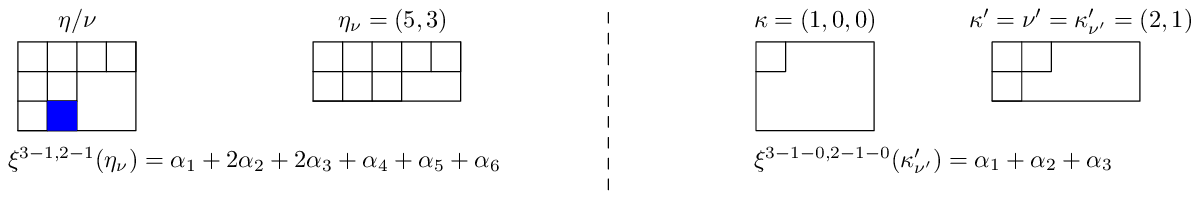}
  
By calculating the remaining terms in the product    by Theorem \ref{thmQPRforComplexGrassmannian}, we have 
  \begin{align*}&\sigma^{(1,1,0)}\star \sigma^{(4, 2, 1)}\\
  =&(\alpha_1+\alpha_2+\alpha_3)(\alpha_1+\cdots+\alpha_6)\sigma^{(4,2,1)}+ (\alpha_1+2\alpha_2+2\alpha_3+\alpha_4+\alpha_5+\alpha_6)\sigma^{(4, 2,2)}\\
    &+(\alpha_1+\cdots+\alpha_6)\sigma^{(4, 3,1)}+\sigma^{(4, 3,2)}+ q\sigma^{(1,1,0)}+q\sigma^{(2, 0,0)}+(\alpha_1+\alpha_2+\alpha_3)q\sigma^{(1,0 ,0)}\\
  \end{align*} 
\end{example}

\begin{cor}\label{corprodcmScQ}
    In $QH^*_T(Gr(m, n+1))$, we have
      \begin{align*}
         \sigma^{1^p}\star \sigma^{(n+1-m,0,\cdots, 0)}&=\sigma^{1^p}\circ\sigma^{(n+1-m,0,\cdots, 0)},\quad \mbox{for } 1\leq p< m;\\
         \sigma^{1^m}\star \sigma^{(n+1-m,0,\cdots, 0)}&=(\alpha_1+\cdots+\alpha_n)\sigma^{(n+1-m, 1, \cdots, 1)}+q.
      \end{align*}
\end{cor}
\begin{proof}
Let  $1\leq p\leq m$ and $\nu=(n+1-m, 0, \cdots, 0)$. It follows directly from    Theorem \ref{thmQPRforComplexGrassmannian}  that
    all possible partitions   are given by  $\eta^{(r)}:=(n+1-m, 1,\cdots, 1, 0,\cdots, 0)$ where $|\eta^{(r)}|=n+1-m+r$, $0\leq r\leq m-1$; and
     the $q$-terms occur only if there exists $\kappa$   satisfying    $p-1\geq r=|\kappa'|= m-1+|\kappa|\geq m-1$.
     Hence, if $p<m$, then $\sigma^{1^p}\star \sigma^{(n+1-m,0,\cdots, 0)}$ involves no $q$-terms. If $p=m$, then $|\kappa|=0$, namely
      $(0, \cdots, 0)$ is the only partition   satisfying the required properties.  Hence,
     $$\sigma^{1^m}\star \sigma^{\nu}=\sum_{r=0}^{m-1}\xi^{m-r, m-r}(\eta^{(r)}_\nu)\sigma^\eta+
   \xi^{m-1-(m-1), m-1-(m-1)}\big((1, \cdots, 1)_{(0, \cdots, 0)}\big)\sigma^{\scriptsize\mbox{id}}q,$$
  in which we note $\xi^{0,0}\big((1, \cdots, 1)_{(0, \cdots, 0)}\big)=1$.
   By definition, we have $\eta^{(r)}_\nu=(n+1-m+r, 0, \cdots, 0)\in\mathcal{P}_{m-r, n+1}$. Hence, $\xi^{m-r, m-r}(\eta^{(r)}_\nu)=0$ unless $r=m-1$. Furthermore when $r=m-1$, we have
    $$\xi^{m-r, m-r}(\eta^{(r)}_\nu)=\xi^{1, 1}((n, 0, \cdots, 0))=s_ns_{n-1}\cdots s_2(\alpha_1)=\alpha_1+\cdots+\alpha_n.$$ Hence, the statement follows.
\end{proof}

\section*{Appendix: Equivariant quantum Giambelli formula for complex Grassmannians}

 We expect out equivariant quantum Pieri rule to have further applications in the equivariant quantum Schubert calculus. To illustrate our expectation, we
 will   reprove  \cite[Theorem \ref{thmequivquanGiam}]{Mihalcea-EQGiambelli}. That is, we will study $QH^*_T(Gr(m, n+1))$, giving  alternative proofs of    the ring presentation   and the equivariant quantum Giambelli formula.
In our approach, we   use the equivariant quantum Pieri rule as in Theorem \ref{thmQPRforComplexGrassmannian}, together with the equivariant Giambelli formula
   \cite{Mihalcea-EQGiambelli,LaRaSa}.
 This is completely similar to  the one given by Buch \cite{Buch-Grassmannian} for the   non-equivariant  quantum  cohomology $QH^*(Gr(m, n+1))$.

 We follow  \cite{Mihalcea-EQGiambelli} for the next facts on equivariant cohomology $H^*_T(Gr(m, n+1))$.
  Treat   $S=\mathbb{Z}[\alpha_1, \cdots, \alpha_n]$ as a subring of $\mathbb{Z}[\mathbf{t}]=\mathbb{Q}[t_1, \cdots, t_{n+1}]$ via
 $$\alpha_i\mapsto t_{n+2-i}-t_{n+1-i}, \quad i=1, \cdots, n.$$
  By convention, we denote $t_i=0$ if $i\leq 0$ or $i\geq n+2$.
 Let $e_i=e_i(x_1, \cdots, x_m; \mathbf{t}), i=1, \cdots, m$   (resp. $h_j=h_j(x_1, \cdots, x_m; \mathbf{t}), j=1,\cdots, \cdots, {n+1-m}$) denote  the elementary  (resp. complete) homogeneous factorial Schur functions.
By convention, we denote $e_0=h_0=1$, $e_i=0$ if $i<0$ or $i>m$, and  $h_j=0$ if $j<0$ or $j>n+1-m$.
Define $\tau^{s}$ inductively by
   $\tau^{0}e_p:=e_p\,\,\,\mbox{ and }\,\,\,\tau^{s}e_p:=\tau^{s-1}e_p+(t_{s}-t_{m-p+s+1})\tau^{s-1}e_{p-1}.$
  Denote by $\lambda^T:=(\lambda_1^T, \cdots, \lambda_{n+1-m}^T)\in \mathcal{P}_{n+1-m, n+1}$ the transpose of a given partition $\lambda\in\mathcal{P}_{m, n+1}$.
%%As we can see below, although every  $\tau^{j-1}e_{1+j-i}$ is   in $H^*_T(Gr(m, n+1))\otimes_S \mathbb{Q}[\mathbf{t}]$, the relevant determinant is in fact an element in $H^*_T(Gr(m, n+1))$.
Let  $H_k:=\det\big(\tau^{j-1}e_{1+j-i}\big)_{1\leq i, j\leq k}.$
We will need the next  lemma, which   follows directly from equation (2.10) of \cite{Mihalcea-EQGiambelli}
%%(i.e., a formula proved in \cite{Macdonald}, I.3, p. 56, Ex. 20(b)) and the   identification between factorial Schur functions and the    Schubert classes $\sigma^{1^p}$, $\sigma^r$ (see e.g. \cite{Mihalcea-EQGiambelli}).
\begin{lemma}\label{lemmavanishcScQequiv}
  For any $M\in\mathbb{Z}^+$, in $H_T^*(Gr(m, n+1))$, we have
     $$ \sum_{p=0}^m(-1)^p\sigma^{1^p}\circ\tau^{1-M}H_{M-p}=0 $$
 with  %$\tau^{-s} H_{j}$  defined inductively by
  $\tau^{0}H_j:=H_j$   and $\tau^{-s}H_j:=\tau^{1-s}H_j+(t_{j+m-s}-t_{1-s})\tau^{1-s}H_{j-1}.$
\end{lemma}

 \begin{thm}[Equivariant quantum Giambelli formula; Theorem 4.2 of \cite{Mihalcea-EQGiambelli}]\label{thmequivquanGiam}
 There is a canonical isomorphism of $S[q]$-algebras,
       $$S[q][e_1, \cdots, e_m]/\langle H_{n-m+2}, \cdots, H_n, H_{n+1}+(-1)^mq\rangle \longrightarrow QH^*_T(Gr(m, n+1)),$$
defined by $e_p\mapsto \sigma^{1^p}$.
 Under this isomorphism,  $\sigma^{r} =H_r$ for $r\leq n+1-m$, and
%every Schubert class $\sigma^\lambda$ in $QH^*_T(Gr(m, n+1))$ can be expressed by
        $$\sigma^{\lambda}=\det\big(\tau^{j-1}e_{\lambda_i^T+j-i}\big)_{1\leq i, j\leq n+1-m}.$$
   \end{thm}

\begin{proof} It is sufficient (1) to calculate $H_k$ and $\det\big(\tau^{j-1}e_{\lambda_i^T+j-i}\big)$ with respect to the equivariant quantum product and (2) to subtract the quantum corrections, by an equivariant quantum extension of    \cite[Proposition 11]{FuPa} (or \cite[Proposition 2.2]{SiTi}).
The known ring presentation of $H_T^*(Gr(m, n+1))$ is read off from the first half of the statement by evaluating $q=0$, and the known equivariant Giambelli formula is exactly of the same form as in the second half.
Thus for any $\lambda\in \mathcal{P}_{m, n+1}$, we have
  $$\det\big(\tau^{j-1}e_{\lambda_i^T+j-i}\big)_{1\leq i, j\leq n+1-m}=\sigma^\lambda+q \cdot g_\lambda\,\,\mbox{ in } QH^*_T(Gr(m, n+1)), $$
for some element $g_\lambda \in   QH^*_T(Gr(m, n+1))$.
 The determinant  $\det\big(\tau^{j-1}e_{\lambda_i^T+j-i}\big)%_{1\leq i, j\leq n+1-m}
 $   in $QH^*_T(Gr(m, n+1))\otimes_S\mathbb{Z}[\mathbf{t}]$
    is a summation of the form
    $$f(\mathbf{t})\sigma^{1^{i_1}}\star \sigma^{1^{i_2}}\star \cdots \star \sigma^{1^{i_{n+1-m}}}.$$
  By     Theorem \ref{thmQPRforComplexGrassmannian} and induction,   the expansion of
      $\sigma^{1^{i_1}}\star \sigma^{1^{i_2}}\star \cdots \star \sigma^{1^{i_{j}}}$ involves no $q$-terms, and    all Schubert classes  in the expansion are of the form  $\sigma^\mu$,   $\mu=(\mu_1, \cdots, \mu_m)$ with $\mu_1\leq j$. Hence, $g_\lambda=0$. That is, the second part of the statement holds, by noting that the determinant lies
        in $QH^*_T(Gr(m, n+1))$.
 In particular, we have $H_r=\sigma^r$ in  $QH^*_T(Gr(m, n+1))$,   for $r=0, 1, \dots,  n-m+1$.

      Clearly,  $\tau^{j-1}e_{1+j-i}$ is zero if $1+j<i$, or of degree $1+j-i$ otherwise. Hence, $\deg H_r=r$.
  Since  $\deg q=\langle c_1(T_{Gr(m, n+1)}), \sigma_{s_k}\rangle=n+1$, it follows that no $q$-term is involved in the expansion of $H_r$ in $QH^*_T(Gr(m, n+1))$, whenever $r<n+1$.

In $H^*_T(Gr(m, n+1))$ it follows from Lemma \ref{lemmavanishcScQequiv} with respect to  $M=n+1$ that
 %%%       $$\tau^{-n}H_{n+1}=(-1)^{m-1}e_m  \tau^{-n}H_{n+1-m}+\sum_{p=1}^{m-1}(-1)^{p-1}e_p \tau^{-n}H_{n+1-p}.$$
 $$H_{n+1}=(-1)^{m-1}e_m  H_{n+1-m}+\sum_{i=0}^{n}f_i(\mathbf{t})H_i+\sum_{j=0}^{n-m}g_{m,j}(\mathbf{t})e_m H_j+\sum_{p=1}^{m-1}\sum_{r=0}^{n-p} g_{p, j}(\mathbf{t})e_{p}H_r$$
 for some $f_i, g_{m, j}, g_{p, r}\in \mathbb{Z}[\mathbf{t}]$. Now we   compute the $q$-terms in the expansion of  the right-hand side as multiplications in $QH^*_T(Gr(m, n+1))\otimes_S\mathbb{Z}[\mathbf{t}]$. With respect to the equivariant quantum multiplications, we have shown $H_r=0$ if $n-m+2\leq r\leq n$, and
 $H_r=\sigma^r$ if $0\leq r\leq n-m+1$.
 Hence,  it follows from Theorem \ref{thmQPRforComplexGrassmannian} (resp. Corollary \ref{corprodcmScQ}) that $\sum_{j=0}^{n-m}g_{m,j}(\mathbf{t})\sigma^{1^m}\star \sigma^p$
 (resp. $\sum_{p=1}^{m-1}\sum_{r=0}^{n-p} g_{p, j}(\mathbf{t})\sigma^{1^p}H_r$)  involves no $q$-terms. Hence, the only $q$-term  in the expansion of $H_{n+1}$
  comes from $$(-1)^{m-1}\sigma^{1^m}  \sigma^{n-m+1}=(-1)^{m-1}((\alpha_1+\cdots+\alpha_n)\sigma^{(n+1-m, 1, \cdots, 1)}+q),$$ by Corollary \ref{corprodcmScQ} again. Hence,
   $H_{n+1}=(-1)^{m-1}q $ in $QH^*_T(Gr(m, n+1))$.
\end{proof}

\end{document}